\documentclass[11pt]{article}
\usepackage[utf8]{inputenc}
\usepackage{amsmath,amssymb,amsthm}
\usepackage{mathtools}
\usepackage{enumerate}
\usepackage{tikz}
\usepackage{lipsum}
\newcommand{\q}[1]{``#1''}

\newtheorem{theorem}{Theorem}[section]
\newtheorem{lemma}[theorem]{Lemma}
\newtheorem{definition}[theorem]{Definition}

\newtheorem{notation}[theorem]{Notation}
\newtheorem{remark}[theorem]{Remark}
\newtheorem{corollary}[theorem]{Corollary}
\newtheorem{theorem/definition}[theorem]{Satz/Definition}
\newtheorem{proposition}[theorem]{Proposition}

\makeatletter
\newtheorem*{rep@theorem}{\rep@title}
\newcommand{\newreptheorem}[2]{%
\newenvironment{rep#1}[1]{%
 \def\rep@title{#2 \ref{##1}}%
 \begin{rep@theorem}}%
 {\end{rep@theorem}}}
\makeatother

\newreptheorem{theorem}{Theorem}

\newcommand{\R}{\mathbb R}

\newcommand{\N}{\mathbb{N}}

\newcommand{\C}{{\cal C}}

\newcommand{\op}{\operatorname}
\newcommand{\supp}{\operatorname{supp}}

\newcommand{\norm}[1]{\left\lVert#1\right\rVert}
\newcommand{\ep}{\varepsilon}

\renewcommand{\)}{\right)}

\newcommand{\Id}{\op{Id}}
\newcommand{\overbar}[1]{\mkern 1.5mu\overline{\mkern-1.5mu#1\mkern-1.5mu}\mkern 1.5mu}

\numberwithin{equation}{section}

\begin{document}

\title{$C_{loc}^{1,1}$ optimal pairs in the dual optimal transport problem for a Lorentzian cost along displacement interpolations}
\author{{\sc Alec Metsch$^{1}$} \\[2ex]
      $^{1}$ Universit\"at zu K\"oln, Institut f\"ur Mathematik, Weyertal 86-90, \\
      D\,-\,50931 K\"oln, Germany \\
      email: ametsch@math.uni-koeln.de \\[1ex]
      }

\maketitle

\begin{abstract}
\noindent
We consider the optimal transportation problem on a globally hyperbolic spacetime with a cost function $c$, which corresponds to the optimal transportation problem on a complete Riemannian manifold where the cost function is given by the squared Riemannian distance. Adapting a proof of Bernard \cite{Bernard} and building upon methods of weak KAM theory, we will establish the existence of $C_{loc}^{1,1}$ optimal pairs for the dual optimal transport problem for probability measures along displacement interpolations.
\end{abstract}

\section{Introduction} 
Let $(M,g)$ be a globally hyperbolic spacetime, and let $h$ be a complete Riemannian metric on $M$. According to Theorem 3 in \cite{Bernard/Suhr}, there exists a smooth manifold $N$ and a splitting $M\cong N\times \R$ such that the mapping $\tau:M\to \R,\ x\cong (z,t)\mapsto t$, satisfies the growth condition
\begin{align}
    d_x\tau(v)\geq \max\{|v|_h,2|v|_g\} \text{ for each } v\in \C_x. \label{splitting}
\end{align}
Here, $|v|_h$ and $|v|_g:=\sqrt{|g_x(v,v)|}$ denote the $h$-norm and the $g$-norm of the vector $v$, respectively, and $\C_x$ is the set of all future-directed causal vectors, including $0$. We now consider the Lagrangian action 
\begin{align*}
L:TM\to \R\cup \{\infty\},\ L(x,v):=
\begin{cases}
\(d_x\tau(v)-|v|_g\)^2, \ &v\in \C_x,
\\
\infty,\ &\text{ otherwise}.
\end{cases}
\end{align*}
The associated family of cost functions $c_t:M\times M\to [0,\infty]$, which arise as the minimal time-$t$-action of the Lagrangian above, is given by
\begin{align}  
c_t(x,y):=
\begin{cases}
\frac 1t(\tau(y)-\tau(x)-d(x,y))^2,\ &(x,y)\in J^+,
\\\\
\infty,\ &(x,y)\notin J^+, \label{haihduaisoafafdfsfdfsfdfdsfsfdsfdsfs}
\end{cases}
\end{align}
see Section \ref{sec2}.
Here, $d$ denotes the Lorentzian distance function (or time separation) and $J^+$ denotes the set of all causally related points in $M\times M$. Let us denote $c:=c_1$.

Now, denoting by ${\cal P}(M)$ the set of Borel probability measures on $M$, let $\mu_0,\mu_1\in {\cal P}(M)$ and assume that the total cost of the optimal transportation problem associated with $c$, 
\begin{align}
C(\mu_0,\mu_1):=\inf\bigg\{\int_{M\times M} c(x,y)\, d\pi(x,y)\mid \pi \in  \Gamma(\mu_0,\mu_1)\bigg\}, \label{bbbb}
\end{align}
is finite. As usual, $\Gamma(\mu_0,\mu_1)$ denotes the set of all couplings between $\mu_0$ and $\mu_1$, i.e.\ all Borel probability measures $\pi$ on $M\times M$ with marginals $\mu_0$ and $\mu_1$. 
Using well-established methods (as in Chapter 7 of \cite{Villani}, see also Lemma \ref{a0}), one can show that if $\pi\in \Gamma_o(\mu_0,\mu_1)$, i.e.\ $\pi\in \Gamma(\mu_0,\mu_1)$ is an optimal coupling in the sense that it minimizes \eqref{bbbb}, then there exists a Borel probability measures 
$\Pi\in {\cal P}(\Gamma)$ ($\Gamma$ denotes the Polish space consisting of all future-directed minimizing curves $\gamma:[0,1]\to M$ for the Lagrangian action $L$, see Corollary \ref{eeef}) that satisfies
\[(e_0,e_1)_\#\Pi=\pi.\]
Here, $(e_0,e_1)_\#\Pi$ denotes the push-forward measure of $\Pi$ w.r.t.\ the Borel map $(e_0,e_1):\Gamma\to M\times M,\ \gamma\mapsto (\gamma(0),\gamma(1))$. One says that $\Pi$ is associated with $\pi$.
The displacement interpolation associated with $\Pi$ is defined as the curve of probability measures 
\[
(\mu_t)_{0\leq t\leq 1} \text{ where }
\mu_t:=(e_t)_\#\Pi.
\]
 Let $0<s<t<1$ be two given intermediate times. We consider the optimal transportation problem $C^{t-s}(\mu_s,\mu_t)$ for the intermediate measures $\mu_s$, $\mu_t$ w.r.t. the cost function
$c_{t-s}$. It is not difficult to prove that $\pi_{s,t}:=(e_s,e_t)_\#\Pi$ constitutes an optimal coupling. Moreover, using methods as in \cite{Villani}, one can show that this coupling is unique (but we won't need this fact).

Next, we consider the dual formulation of optimal transport. For the probability measures $\mu_0$ and $\mu_1$, this takes the form
\begin{align}
C(\mu_0,\mu_1)= \sup\bigg\{\int_M \psi(y)\, d\mu_1(y)-\int_M \varphi(x)\, d\mu_0(x)\bigg\},\label{z}
\end{align}
where the supremum is taken over all functions $\psi\in L^1(\mu_1)$ and $\varphi\in L^1(\mu_0)$ satisfying the constraint $\psi(y)-\varphi(x)\leq c(x,y)$ for all $x,y\in M$ (see \cite{Ambrosio}, Theorem 6.1.1).
A pair $(\varphi,\psi)\in L^1(\mu_0) \times L^1(\mu_1)$ is said to be optimal if it maximizes the expression on the right-hand side of the above identity. 

Our main theorem in this paper deals with the existence of a $C^{1,1}_{loc}$ optimal pair for intermediate times. More precisely, denoting by $I^+$ the set of all chronologically related points, we have:

\begin{theorem}\label{main}
Let $\mu_0,\mu_1\in {\cal P}(M)$ and assume that $C(\mu_0,\mu_1)$ is finite. Let $\pi \in \Gamma_o(\mu_0,\mu_1)$ and $\Pi$ be a dynamical optimal coupling associated with $\pi$. Fix $0<s<t<1$.   Suppose that $(\varphi,\psi)$ is an optimal pair in \eqref{z} and that $\pi(I^+)=1$. Then, there exists an optimal pair $(\Phi_s,\Psi_t)$ for the dual formulation of the optimal transportation problem for the measures $\mu_s$ and $\mu_t$ associated with the cost function $c_{t-s}$, i.e.
	\begin{align*}
	C^{t-s}(\mu_s,\mu_t)=\int_M \Psi_t(y)\, d\mu_t(y)-\int_M \Phi_s(x)\, d\mu_s(x),
	\end{align*}
	such that $\Phi_s$ is $C_{loc}^{1,1}$ on an open set of full $\mu_s$-measure and $\Psi_t$ is $C_{loc}^{1,1}$ on an open set of full $\mu_t$-measure.
\end{theorem}

The assumption $\pi(I^+)=1$ implies that the dynamical optimal coupling $\Pi$ is concentrated on the subset of timelike minimizers. This property is crucial, and we will explain its significance in more detail later. Note that, based on \cite{Kell}, in \cite{Metsch}, conditions on the measures $\mu_0$ and $\mu_1$ were established to ensure that any optimal coupling is concentrated on $I^+$.

Obviously, any optimal pair $(\varphi,\psi)$ in \eqref{z} satisfies
\[
\psi(y)-\varphi(x)=c(x,y) \text{ for $\pi$-a.e.\ } (x,y).
\]
This leads to the definition of a calibrated pair, adapted (with a slightly different definition) from \cite{Fathi/Figalli}. 
\begin{definition}\rm
If $\pi$ is an optimal coupling (between its marginals), a pair of functions $\varphi,\psi:M\to \R\cup\{\pm\infty\}$ is called \emph{$(c,\pi)$-calibrated} if
$\psi(y)-\varphi(x)\leq c(x,y)$ for all $(x,y)$ and 
\begin{align}
\psi(y)-\varphi(x)=c(x,y) \text{ for $\pi$-a.e. } (x,y). \label{f1}
\end{align}
Here, we use the convention $\pm\infty\mp\infty :=-\infty$.
\end{definition}

Due to the relatively explicit construction of the optimal pair $(\Phi_s,\Psi_t)$ in the proof of Theorem \ref{main}, Theorem \ref{main} follows from the following intermediate result.

\begin{theorem}\label{main2}
Let $\mu_0,\mu_1\in {\cal P}(M)$ and assume that $C(\mu_0,\mu_1)$ is finite. Let $\pi \in \Gamma_o(\mu_0,\mu_1)$ and $\Pi$ be a dynamical optimal coupling associated with $\pi$. Fix $0<s<t<1$. Suppose that the pair $(\varphi,\psi)$ is $(c,\pi)$-calibrated and that $\pi(I^+)=1$. Then, there exists a $(c_{t-s},\pi_{t-s})$-calibrated pair $(\Phi_s,\Psi_t)$ such that $\Phi_s$ is $C_{loc}^{1,1}$ on an open set of full $\mu_s$-measure and $\Psi_t$ is $C_{loc}^{1,1}$ on an open set of full $\mu_t$-measure.
\end{theorem}
The proof follows an approach inspired by weak KAM theory, specifically the work of Bernard \cite{Bernard}. There, the author proves the existence of a $C^{1,1}$ sub-solution to the Hamilton-Jacobi equation
\[
H_T(x,d_xu)=c
\]
for the critical value $c=c(H_T)$, where $H_T$ is a Tonelli Hamiltonian on a compact manifold $M_T$. Let $L_T:TM\to \R$ denote the Lagrangian associated with $H_T$. Recall that the forward and backward Lax-Oleinik evolutions of an arbitrary function $u:M_T\to \overbar \R$ are defined as
\begin{align*}
   &T_{t}u(x):=\inf\{u(y)+h_t(y,x)\mid y\in M_T\},\
   \\[10pt]
   &\hat T_{t}u(x):=\sup\{u(y)-h_t(x,y)\mid y\in M_T\}, 
\end{align*}
where $h_t(x,y)$ denotes the minimal time-$t$-action joining $x$ to $y$. Since $T_{t}$ and $\hat T_{t}$ preserve the set of sub-solutions, it follows that if $u$ is a sub-solution, then $\hat T_s(T_tu)$ is also a sub-solution for any $s,t>0$. 
Moreover, it is well-known that if $v$ is continuous, then $T_tv$ and $\hat T_tv$ are locally semiconcave and locally semiconvex, respectively (see \cite{Fathi}). Thus, $\hat T_s(T_tu)$ is locally semiconvex. The key challenge in \cite{Bernard} is to show that the local semiconcavity of $T_tu$ \q{survives} the backward regularization $\hat T_s(T_tu)$ for small values of $s$, so that the resulting function remains both locally semiconcave and locally semiconvex. This, in turn, implies that $\hat T_s(T_tu)$ is $C^{1,1}$ (see \cite{Fathi}).
 Figalli \cite{FigalliPHD}, in his PhD thesis, and Fathi, Figalli and Rifford \cite{Fathi/Figalli/Rifford} adapted this method to the setting of a non-compact Tonelli Hamiltonian system. 

To demonstrate how we intend to exploit this method, we begin with the following well-known fact, which is easy to verify (\cite{Villani}, Theorem 7.36):

If $(\varphi,\psi)$ is a $(c,\pi)$-calibrated pair (recall that $\pi\in \Gamma_o(\mu_0,\mu_1)$), then the following functions
  \begin{align*}
    &\varphi_s(x):=\inf_{y\in M} \varphi(y)+c_s(y,x) \text{ and }
    \\[10pt]
    &\psi_t(y):=\sup_{x\in M} \psi(x)-c_{1-t}(y,x)
\end{align*}
form a $(c_{t-s},\pi_{s,t})$-calibrated pair (see also Lemma \ref{jjjjj} and Lemma \ref{h1}). Notice that $\varphi_s$ and $\psi_t$ evolve according to the forward and backward Lax-Oleinik semigroups $T_s$ and $\hat T_{1-t}$, respectively, where we replaced the minimal action $h_t$ in the definition  with the corresponding Lorentzian action. That is, 
\begin{align*}
    \varphi_s=T_s\varphi \text{ and } \psi_t =\hat T_{1-t} \psi.
\end{align*}
Following the ideas in \cite{Bernard}, the first attempt would be to define, for small $\tau>0$,
\begin{align}
    \Phi_s:=\hat T_\tau(T_\tau\varphi_s)=\hat T_\tau T_{s+\tau}\varphi \text{ and } \Psi_t:=T_\tau(\hat T_\tau \psi_t)=T_\tau\hat T_{1-t+\tau}\psi. \label{y11}
\end{align}
Note that the times, over which we regularize the function $\varphi_s$ (and analogously for $\psi_t$) first forward and then backward must be equal, so that the resulting pair $(\Phi_s,\Psi_t)$ is indeed $(c_{t-s},\pi_{s,t})$-calibrated, see Lemma \ref{h1}.

However, unfortunately, this approach does not work out. One issue arises from the lack of regularity results for $\varphi$ or $\varphi_{s+\tau}$. This is not a problem in \cite{Bernard} or in \cite{FigalliPHD}, as sub-solutions of the Hamilton-Jacobi equation for Tonelli Hamiltonians are always Lipschitz continuous. However, there are even results concerning the regularity of the Lax-Oleinik evolution for arbitrary functions in the case of possibly non-compact Tonelli systems \cite{FathiHJ}:

If $u:M_T\to \overbar \R$ is any function and $T_{t_0}u(x_0)$ is finite at some $(t_0,x_0)\in (0,\infty)\times M_T$, then $T_{t}u$ is locally semiconcave on $(0,t_0)\times M_T$. 

This result is sufficient to prove that Theorem 1.1 and Theorem 1.3 hold in the case of Tonelli Lagrangian systems with $\Phi_s$ and $\Psi_t$ given by \eqref{y11}.\footnote{Of course we mean the corresponding variants of the theorems. In particular, there is no counterpart for the assumption $\pi(I^+)=1$, and we have to replace $c$ with $h_1$, the minimal time-$1$-action.} Indeed, arguing for $\Psi_t$, we first see that the above mentioned result shows that $\Psi_t=T_\tau \hat T_{1-t+\tau}\psi$ is locally semiconcave on $M_T$, as $T_{1-t+\tau}\hat T_{1-t+\tau}\psi(y)$ is finite, where $y\in M_T$ is such that $\psi(y)-\varphi(x)=h_1(x,y)$. Recall that $h_1$ is the minimal time-$1$-action and such a $y$ exists thanks to the $(h_1,\pi)$-calibration. Moreover, since the functions $(\hat T_{1-t+\tau}\psi)_{\tau}$ are locally equi-Lipschitz and uniformly locally semiconvex for small $\tau$ thanks to the above result (applied to the backward semigroup), Appendix B in \cite{Fathi/Figalli/Rifford} shows that, for sufficiently small $\tau>0$, $\Psi_t$ is also locally semiconvex. Hence, $\Psi_t$ is $C_{loc}^{1,1}$.

Unfortunately, these methods cannot be applied to our case, at least not without additional reasoning. The issue, of course, stems from the lack of regularity of the Lagrangian $L$ compared to a Tonelli Lagrangian. It is easy to construct a counterexample in which the mentioned result of \cite{FathiHJ} fails in the Lorentzian case. 

Our attempt is the following: Since the dynamical optimal coupling $\Pi$ must be concentrated on timelike minimizers (because $\pi(I^+)=1$), there exists a Borel set $A\subseteq \Gamma$, consisting of timelike minimizers $\gamma:[0,1]\to M$ such that $(\gamma(0),\gamma(1))$ satisfies the identity in \eqref{f1}, with $\Pi(A)=1$. By definition, $\mu_t$ is concentrated on the set 
\begin{align*}
A_t:=\{\gamma(t)\mid \gamma \in A\}.
\end{align*}
In a first step we prove that, if $\gamma \in A$, then $\varphi_s$ and $\psi_t$ are locally semiconcave and locally semiconvex on small open neighbourhoods of $\gamma(s)$ and $\gamma(t)$, respectively. Note that the timelikeness of $\gamma$ is crucial. With the help of calibrated curves (playing the role of $\gamma$), in the proof we adapt and extend the methods of \cite{FathiHJ} concerning regularity of the Lax-Oleinik evolution to the Lorentzian case. In particular, $\varphi_s$ and $\psi_t$ are locally semiconcave and locally semiconvex on small open neighbourhoods of $A_s$ and $A_t$, respectively. 
The same proof (with calibrated curves $\gamma_{|[0,s+\tau]}$ and $\gamma_{|[t-\tau,1]}$) shows that $\hat T_\tau \varphi_{s+\tau}$ and $T_\tau\psi_{t-\tau}$ are locally semiconvex and locally semiconcave on open neighbourhoods of $A_s$ and $A_t$, respectively.

In a second step, we adapt the proof of \cite{Bernard} and prove that, if $\gamma \in A$, then for small $\tau>0$ the functions
\begin{align*}
    \hat T_\tau T_{s+\tau}\varphi \text{ and } T_\tau\hat T_{1-t+\tau}\psi
\end{align*}
constitute a $(c_{t-s},\pi_{s,t})$-calibrated pair and are also locally semiconcave and locally semiconvex on open neighbourhoods of $\gamma(s)$ and $\gamma(t)$, respectively. In particular, both functions are $C^{1,1}_{loc}$ on these sets. Note that $\tau$ depends on $\gamma$. However, with a partition of unity argument, we will show how to obtain the stated $(c_{t-s},\pi_{s,t})$-calibrated functions $\Phi_s$, $\Psi_t$ which are $C_{loc}^{1,1}$ on open neighbourhoods of $A_s$ and $A_t$, respectively.
\bigskip

This paper is organized as follows: In Section 2, we investigate the existence, regularity and compactness properties of minimizing curves for the Lagrangian $L$, as well as various properties of the minimal action and $L$ itself. In Section 3 we introduce and study the Lax-Oleinik evolution of a function (playing the role of $\varphi$ or $\psi$) with the assumption that a calibrated curve (corresponding to $\gamma \in A$) exists. We will prove the first step as mentioned above, i.e.\ the local semiconcavity and local semiconvexity of $\varphi_s$ and $\psi_t$ on open neighbourhoods of $A_s$ and $A_t$, respectively. Section 4 deals with the regularity of the Lax-Oleinik evolution of a $C^1$ Lipschitz function, whose gradient is not future-directed causal. This study will be needed in Section 5, where we are in position to extend the results obtained by Bernard to the Lorentzian case, and where we prove our main Theorems \ref{main} and \ref{main2}.

\section{The Lagrangian} \label{sec2}

From now on until the appendix, let $(M,g)$ always denote a globally hyperbolic spacetime. We refer to the appendix for a brief review of the basic concepts of Lorentzian geometry. Let $h$ be any complete Riemannian metric, and let $\tau$ be a function satisfying \eqref{splitting}. We study the Lagrangian
\begin{align*}
    L:TM\to\R\cup \{\infty\},\ L(x,v):=
    \begin{cases}
        (d_x\tau(v)-|v|_g)^2,\ & \text{if } v\in \C_x,
        \\\\
        +\infty,\ & \text{else.}
    \end{cases}
\end{align*}
Note that $L$ is smooth only on $\op{int}(\C)$, where $\C:=\{(x,v)\in TM\mid v\in \C_x\}$. To begin, we introduce some basic notations and results. 

\begin{notation}\rm
    If not otherwise specified, a curve $\gamma:[a,b]\to M$ is always assumed to be absolutely continuous. When we say that $\gamma$ is causal, we always mean future-directed causal, i.e.\ $\dot \gamma(t)\in \C_{\gamma(t)}$ for a.e.\ $t\in [a,b]$. We use the notion timelike only for $C^1$-curves and always mean future-directed timelike, i.e.\ $\dot \gamma(t)\in \op{int}(\C_{\gamma(t)})$ for all $t\in [a,b]$.
\end{notation}

\newcommand{\LL}{\mathbb{L}}

\begin{definition}\rm
The \emph{action} of a curve $\gamma:[a,b]\to M$ is defined by
\[
\LL(\gamma):=\int_a^b L(\gamma(t),\dot \gamma(t))\, dt\in [0,\infty].
\]
Observe that if $\LL(\gamma)$ is finite, then $\gamma$ must be causal (recall that $0$ is causal in our definition). The converse implication does not hold.
\end{definition}
\begin{definition}\rm
\begin{enumerate}[(a)]
\item
    A curve $\gamma:[a,b]\to M$ is called a \emph{minimizer} if, for any other curve $\tilde \gamma:[a,b]\to M$ with the same start- and endpoint as $\gamma$, we have
    \begin{align*}
        \mathbb{L}(\gamma)\leq \mathbb{L}(\tilde \gamma).
    \end{align*}
    \item For $t>0$, we denote by $\Gamma^t$ the set of all causal minimizers $\gamma:[0,t]\to M$.
    \end{enumerate}
\end{definition}

\begin{lemma}
    The second derivative along the fibers, $\frac{\partial^2 L}{\partial v^2}(x,v)$, is positive definite at each point $(x,v)\in \op{int}(\C)$.
\end{lemma}
\begin{proof}
    This is an easy calculation that was already carried out in \cite{Metsch}, Corollary 7.17.
\end{proof}

\begin{remark}\rm
Every timelike minimizer $\gamma:[a,b]\to M$ satisfies, in local coordinates, the Euler-Lagrange equation
\[
\frac{d}{dt}\frac{\partial L}{\partial v}(\gamma(t),\dot \gamma(t))=\frac{\partial L}{\partial x}(\gamma(t),\dot \gamma(t)),\ t\in[a,b].
\]
Observe that this holds only for timelike curves since $L$ is not differentiable on the entire set $\C$.
     Since the second derivative along the fibers of $L$ is positive definite on $\op{int}(\C)$, it follows that there exists a smooth (local) Euler-Lagrange flow $\phi_t$ on $\op{int}(\C)$. It is well-known that the speed curve $(\gamma(t),\dot \gamma(t))$ of any timelike minimizer is part of an orbit of the flow, cf. \cite{Fathi}.
\end{remark}

\begin{lemma}\label{minimizer}
    A causal curve $\gamma:[a,b]\to M$ is a minimizer if and only if $\gamma$ is a maximizer for the Lorentzian lenght functional and $L(\gamma(t),\dot \gamma(t))$ is constant.

    In particular, every minimizer is smooth, and for any $(x,y)\in J^+$ and $t>0$, there exists a minimizer connecting $x$ with $y$ in time $t$.
\end{lemma}
\begin{proof}[Sketch of proof]
  For a detailed proof, we refer the reader to \cite{Metsch}, Proposition 3.6. Consider the Lagrangian
    \begin{align*}
    L':TM\to\R\cup \{\infty\},\ L'(x,v):=
    \begin{cases}
        d_x\tau(v)-|v|_g,\ & \text{if } v\in \C_x,
        \\\\
        +\infty,\ & \text{else.}
    \end{cases}
\end{align*}
Using H\"older's inequality, one can easily check that a causal curve $\gamma:[a,b]\to M$ is a minimizer for $L$ if and only if it is a minimizer for $L'$ and $L(\gamma(t),\dot \gamma(t))$ is constant a.e. 

However, since 
\[
\LL'(\gamma)=\tau(\gamma(b))-\tau(\gamma(a))-\ell_g(\gamma)
\]
($\ell_g$ denoting the Lorentzian length), it is obvious that $\gamma$ is a minimizer for $L'$ if and only if $\gamma$ maximizes the Lorentzian length functional. This proves the first part of the lemma.

Since any two points $(x,y)\in J^+$ can be joined by a maximizing geodesic (see \cite{ONeill}, Chapter 14, Proposition 19), which can be reparametrized smoothly such that $L(\gamma(t),\dot \gamma(t))$ is constant, it follows that there always exists a minimizer (for $L$) connecting $x$ with $y$. 

The fact that every minimizer is smooth simply follows from the well-known fact that any maximizing curve for the Lorentzian length functional is a pregeodesic (\cite{Minguzzi}, Theorem 2.12). Hence, every minimizer $\gamma$ is a reparametrization of a geodesic such that $L(\gamma(t),\dot \gamma(t))$ is constant. Thus, $\gamma$ is smooth.
\end{proof}

\begin{definition}\rm
The \emph{minimal time-$t$-action} to go from $x$ to $y$ in time $t$ is defined as
\begin{align*}
   c_t(x,y)&:=\inf\bigg\{\mathbb{L}(\gamma)\mid  \gamma:[0,t]\to M \text{ curve, }\gamma(0)=x,\ \gamma(t)=y\bigg\}
   \\[10pt]
   &= \begin{cases}
       \frac 1t (\tau(y)-\tau(x)-d(x,y))^2,\ & \text{if } y\in J^+(x),
       \\\\
       +\infty,\ & \text{else.}
   \end{cases}
\end{align*}
\end{definition}
\begin{proof}[Proof of the second equality]
    Follows from the preceding lemma.
\end{proof}

The following proposition was proved in \cite{Metsch}, Proposition 3.9:

\begin{proposition}\label{flow}
 There exists a relatively open set ${\cal D}_L\subseteq \C\times \R$ and a continuous (local) flow
 \begin{align*}
     \phi:{\cal D}_L\to \C\subseteq TM,\ (x,v,t)\mapsto \phi_t(x,v),
 \end{align*}
 such that the following properties hold:
 \begin{enumerate}[(i)]
 \item For any $(x,v)\in \C$, the map $\{t\in \R\mid (x,v,t)\in {\cal D}_L\}\to TM,\ t\mapsto \phi_t(x,v),$ is smooth and of the form $(\gamma(t),\dot \gamma(t))$. 
 \item 
 If $x\leq y$ and $\gamma:[a,b]\to M$ is a minimizer connecting $x$ with $y$, then $\gamma$ is part of an orbit of this flow, i.e.\ if $t,s\in [a,b]$ we have $(\gamma(t),\dot \gamma(t))=\phi_{t-s}(\gamma(s),\dot \gamma(s))$.
 \item 
 For any $x\in M$ and $0\neq v\in \C_x$, the curve $t\mapsto \pi \circ \phi_t(x,v)$ is inextendible, i.e.\ if $I$ is the domain of definition of $\phi_\cdot(x,v)$, then the limits $\phi_t(x,v)$ as $t\to \sup{I}$ or $t\to \inf{I}$ do not exist. Here, $\pi:TM\to M$ denotes the canonical projection.
 \item The flow, when restricted to the interior
 $\op{int}(\C),$  is the Euler-Lagrange flow. This means that if $(x,v)\in \op{int}(\C)$,  the orbits of both the flow $\phi$ and the Euler-Lagrange flow passing through $(x,v)$ at time $t=0$ have the same domain of definition and agree throughout that domain.
 \end{enumerate}
 \end{proposition}
 \begin{proof}
     For a detailed proof (except for (iii) and (iv)), we refer the reader to \cite{Metsch}. The proof is not difficult. Indeed, one reparametrizes the geodesic flow on $M$ according to Lemma \ref{minimizer}. Then one has to check all the stated properties above. The proofs of (iii) and (iv) are easy.  
 \end{proof}

\begin{remark}\rm
Because of part (iv) of the above proposition we will denote both the Euler-Lagrange flow as well as the flow from the proposition with $\phi$.
\end{remark}

\begin{corollary}\label{eeef}
The set of causal minimizers $\Gamma^t$ is a closed subset of $C([0,t],M)$. In particular, $\Gamma^t$ is a complete and separable metric space.
\end{corollary}
\begin{proof}
Let $(\gamma_k)\subseteq \Gamma^t$ be a sequence converging to $\gamma \in C([0,t],M)$. 
We need to show that $\gamma \in \Gamma^t$.

Since $\gamma_k$ are minimizers, we have $\gamma_k(s)=\pi\circ \phi_s(x_k,v_k)$ for $x_k:=\gamma_k(0)$ and $v_k:=\dot \gamma_k(0)$. As $\gamma_k\to \gamma$, we have $(x_k,\gamma_k(t))\to (\gamma(0),\gamma(t))$ and the latter lies in $J^+$, since $J^+$ is closed. By continuity of $c_t$ on $J^+$, we obtain
\begin{align*}
C:=\sup_{k\in \N}c_t(x_k,\gamma_k(t))<\infty.
\end{align*}
Now, using \eqref{splitting}, we see that, for each $k\in \N$, we have
\begin{align*}
|v_k|_h\leq d\tau(v_k)\leq 2 L(x_k,v_k)^{1/2}=2t^{-1/2}c_t(x_k,\gamma_k(t))^{1/2}\leq 2t^{-1/2}C^{1/2},
\end{align*}
where we have used that $L(\gamma_k(s),\dot \gamma_k(s))$ is constant.
Thus, the sequence $(x_k,v_k)$ is relatively compact in $TM$, and for a subsequence that we do not relabel, $(x_k,v_k)\to (x,v)\in \C$ as $k\to \infty$. By continuity of the flow, we clearly have
\begin{align*}
    \pi \circ \phi_s(x,v)=\gamma(s) \text{ for all } s\in [0,t] \text{ such that the first curve is defined.}
\end{align*}
If $v=0$, then the first curve is defined for all $s\in \R$ due to Proposition \ref{flow}(ii). If $v\neq 0$, then $s\mapsto \pi \circ \phi_s(x,v)$ is inextendible. Thus, $\phi_s(x,v)$ must be defined for all $s\in [0,t]$ and $\gamma(s)=\pi \circ \phi_s(x,v)$ for all these $s$. In particular, $\gamma$ is smooth, and due to the continuity of $\phi$, it holds $\gamma_k\to \gamma$ in $C^1([0,t],M)$. Then we conclude the proof as follows:
By the continuity of $c_t$ on $J^+$,
\begin{align*}
    c_t(x,\gamma(t))=\lim_{k\to \infty} c_t(x_k,\gamma_k(t))= \lim_{k\to \infty} \LL(\gamma_k)=\LL(\gamma).
\end{align*}
Thus, $\gamma$ minimizes $\LL$ and, hence, belongs to $\Gamma^t$. 

In particular, $\Gamma^t$ is complete and separable as $C([0,1],M)$ is a complete separable space.
\end{proof}

\begin{corollary}\label{ee}
    Let $(x_k)$ and $(y_k)$ be sequences with $y_k\in J^+(x_k)$ converging to $x$ and $y$, respectively. Let $t_k>0$ with $t_k\to t\in (0,\infty)$. For each $k$, suppose that $\gamma_k:[0,t_k]\to M$ is a minimizer connecting $x_k$ with $y_k$. Then, for a subsequence, $\gamma_k$ can be extended to smoothly to $[0,t]$ (if $t_k<t$) and $\gamma_k:[0,t]\to M$ converges in the $C^1$-topology to a minimizing curve $\gamma:[0,t]\to M$ connecting $x$ with $y$. 
\end{corollary}
\begin{proof}
    The same proof as above shows that, for $k\in \N$ and $s\in [0,t_k]$,
    \begin{align*}
|\dot \gamma_k(s)|_h\leq d\tau(\dot \gamma_k(s))\leq 2 L(\gamma_k(0),\dot \gamma_k(s))^{1/2}
&=2t_k^{-1/2}c_{t_k}(\gamma_k(0),\gamma_k(t_k))^{1/2}
\\[10pt]
&\leq 2\sup_{k\in \N} t_k^{-1/2}c_{t_k}(x_k,y_k)^{1/2}.
\end{align*}
In particular, the sequence of curves $\gamma_k$ is equi-continuous. 

Now, let $\ep>0$. Since the curves are equi-continuous and $(\gamma_k(0))=(x_k)$ is relatively compact, the Arzelà-Ascoli theorem gives that $(\gamma_k)$ is precompact in $C([0,t-\ep],M)$. The preceding lemma and its proof tell us that, along a subsequence that we do not relabel, $\gamma_k$ converges in the $C^1$-topology to a minimizing curve $\gamma:[0,t-\ep]\to M$. 

Iterating this proof for the sequence $\ep_m=1/m\to 0$ and extracting a diagonal sequence, it follows that there exists a curve $\gamma:[0,t)\to M$ such that $\gamma_k$ converges to $\gamma$ in the $C^1$-topology on any compact subinterval. In particular, thanks to Proposition \ref{flow},
\begin{align*}
    \gamma(s)=\pi \circ \phi_s(x,\dot \gamma(0)) \text{ for } s\in [0,t) \text{ and } \gamma_k(s)=\pi \circ \phi_s(x_k,\dot \gamma_k(0)) \text{ for } s\in [0,t_k].
\end{align*}
However, since the sequence $\gamma_k$ is equi-continuous, it follows that $\gamma$ is uniformly continuous, so that $\gamma$ can be extended to $[0,t]$. Thus, by Proposition \ref{flow}(iii), the orbit $\phi_s(x,\dot \gamma(0))$ is defined up some $T>t$. Since ${\cal D}_L$ is open, also $\phi_s(x_k,\dot \gamma_k(0))$ is defined for all $s\in [0,t]$ for large $k$, so that $\gamma_k$ can indeed be extended smoothly on $[0,t]$. Clearly, $\gamma_k$ converges to $\gamma$ in the $C^1([0,t],M)$-topology. Using the continuity of $c$ as a map from $(0,\infty)\times J^+$ and the fact that $\gamma_k:[0,t_k]\to M$ is minimizing, it easily follows that also $\gamma$ is a minimizer. Furthermore, it is obvious that $\gamma(t)=y$.
\end{proof}

Exploiting the fact that the Lorentzian distance is locally semiconvex on $I^+$ (\cite{McCann2}, Proposition 3.4) and, as a consequence, the fact that $c=c_1$ is locally semiconvex on $I^+$, the following theorem extends a well-known result from the theory of Tonelli Lagrangians (see e.g.\ \cite{FathiHJ}, Proposition 4.19) to our setting. We keep the proof and the notation similar. Do not confuse the following notation ${\cal C}$ with the the future cone. 
For definitions and further properties on semiconcavity, see Subsection \ref{semiconcavity} in the appendix.

\begin{theorem} \label{A}
    The mapping
    \[
    {\cal C}:(0,\infty)\times I^+,\ (t,x,y)\mapsto c_t(x,y),
    \]
    is locally semiconcave. If $\gamma:[0,t]\to M$ is a minimizing curve with $\gamma(0)=x$ and $\gamma(t)=y$, then 
    \begin{align}
    \left(\partial_t c_t(x,y),-\frac{\partial L}{\partial v}(x,\dot \gamma(0)),\frac{\partial L}{\partial v}(y,\dot \gamma(t)\right)\in \partial^+{\cal C}(t,x,y). \label{mnbvcdfghj}
    \end{align}
    Moreover, the set of all super-differentials is given by the closure of the convex hull of covectors of the form \eqref{mnbvcdfghj}.
   In particular, ${\cal C}$ is differentiable at some point $(t,x,y)$ if and only if there is a unique minimizing curve connecting $x$ with $y$ in time $t$.
\end{theorem}
\begin{proof}
    We have
    $c_t(x,y)=c(x,y)/t$ for $(x,y)\in I^+$ and $t>0$. Moreover, by Lemma \ref{minimizer}, minimizing curves for $c_t$ are precisely affine reparametrisations of the form $[0,t]\ni s\mapsto \gamma(s/t)$ of minimizing curves $\gamma$ for $c$. Hence, it suffices to check that $c$ is locally semiconcave on $I^+$, that
     \begin{align}
     \partial^+{ c}(x,y)=\overline{\op{conv}\bigg\{\left(-\frac{\partial L}{\partial v}(x,\dot \gamma(0)),\frac{\partial L}{\partial v}(y,\dot \gamma(1))\right)\bigg\}}, \label{ffffff}
    \end{align}
    (where we take all minimizing curves $\gamma:[0,1]\to M$ connecting $x$ with $y$),
    and that $c$ is differentiable at $(x,y)\in I^+$ if and only if there exists a unique minimizer connecting $x$ with $y$ (in time $1$).

Using the explicit formulation for $c$, the local semiconvexity follows easily from the well-known fact that the Lorentzian distance function $d$ is locally semiconvex on $I^+$ (see \cite{McCann2}, Proposition 3.4). An elementary proof, based on the methods in \cite{Fathi/Figalli}, can be found in \cite{Metsch}, Proposition 7.13. Moreover, as a byproduct of the proof in \cite{Metsch}, one has
\begin{align}
     \left(-\frac{\partial L}{\partial v}(x,\dot \gamma(0)),\frac{\partial L}{\partial v}(y,\dot \gamma(1))\right)\in \partial^+c(x,y)
     \label{fffff}
\end{align}
for all minimizing curves $\gamma:[0,1]\to M$ connecting $x$ with $y$. Since $\partial^+c(x,y)$ is easily seen to be closed and convex, the inclusion $\supseteq$ in \eqref{ffffff} follows.
     Next, recall the well-known fact that the set of super-differentials of a locally semiconcave function is the closure of the convex hull of reaching gradients (see \cite{Philippis}, Appendix A). This means that $\partial^+ c(x,y)$ is the closure of the convex hull of all covectors $(p,q)\in  T_x^*M\times T_y^*M$ for which there exists a sequence of differentiability points $(x_k,y_k)\in  I^+$ for $c$ that converge to $(x,y)\in I^+$ and such that
     \[
     (p,q)=\lim_{k\to \infty} d_{(x_k,y_k)}c \text{ in the topology of the cotangent bundle.}
     \]
     Let $\gamma_k:[0,1]\to M$ be minimizing (necessarily timelike) curves connecting $x_k$ with $y_k$. Since $c$ is differentiable at $(x_k,y_k)$, there must be a unique super-differential, so that \eqref{fffff} gives
     \[
     d_{(x_k,y_k)}c=\left(-\frac{\partial L}{\partial v}(x_k,\dot \gamma_k(0)),\frac{\partial L}{\partial v}(y_k,\dot \gamma_k(1))\right).
     \]
    Since the sequences $(x_k)$ and $(y_k)$ are relatively compact, it follows from the above corollary that, up to some subsequence, $\gamma_k$ converges in the $C^1$-topology to some minimizing curve $\gamma:[0,1]\to M$. In particular, $\gamma(0)=x$ and $\gamma(1)=y$ and 
     \[
     \dot \gamma(0)=\lim_{k\to \infty} \dot \gamma_k(0),\ \dot \gamma(1)=\lim_{k\to \infty} \dot \gamma_k(1).
     \]
     Since $(x,y)\in I^+$, $\gamma$ is necessarily timelike. Therefore, 
      \[
     (p,q)=\lim_{k\to \infty} \left(-\frac{\partial L}{\partial v}(x_k,\dot \gamma_k(0)),\frac{\partial L}{\partial v}(y_k,\dot \gamma_k(1))\right)=
     \left(-\frac{\partial L}{\partial v}(x,\dot \gamma(0)),\frac{\partial L}{\partial v}(y,\dot \gamma(1))\right).
     \]
     This proves \eqref{ffffff}. Since differentiability is equivalent to unique super-differentiability, the last part of the theorem follows from \eqref{ffffff} and the fact that, if $\gamma_1$ and $\gamma_2$ are two minimizing curves connecting $x$ to $y$ and ${\cal L}(x,\dot \gamma_1(0))={\cal L}(x,\dot \gamma_2(0))$, then $\dot \gamma_1(0)=\dot \gamma_2(0)$ since ${\cal L}$ is a diffeomorphism. Thus, since both minimizers $\gamma_1$ and $\gamma_2$ have the same position and the same speed at $t=0$, we must have $\gamma_1=\gamma_2$.
\end{proof}

\begin{remark}\rm
 Since minimizing curves are precesily certain reparametrizations of maximizing geodesics, it follows that ${\cal C}$ is differentiable at some $(t,x,y)$ if and only if there is a unique maximizing geodesic between $x$ and $y$.   
\end{remark}

\begin{definition}\rm
The \emph{dual future cone} $\C_x^*\subseteq T_x^*M$ is defined as the image of $\C_x$ under the canonical isomorphism
\begin{align*}
    T_xM\to T_x^*M,\ v\mapsto g(v,\cdot).
\end{align*}
    We also call
       $\C_*:=\{(x,v)\in T^*M\mid v\in \C_x^*\}$
    the dual future cone.
\end{definition}

\begin{remark}\rm \label{dddd}
   It is not hard to check that $\op{int}(\C_x^*)$ is the image of $\op{int}(\C_x)$ under the  the canonical isomorphism given above. Here, we take the interior w.r.t to the subspace topology in $T^*_xM$ resp. $T_xM$. 

   It is easy to see that the dual future cone $\C_x^*$ (resp.\ $\op{int}(\C_x^*)$) is characterized by the fact that $pv\leq 0$ (resp.\ $<0$) for all $v\in \C_x$.
\end{remark}

\begin{definition}\rm
The \emph{Legendre transform} of $L$ is defined as the map
\begin{align*}
   {\cal L}:\op{int}(\C) \to T^*M,\ (x,v)\mapsto \left(x,\frac{\partial L}{\partial v}(x,v)\right).
\end{align*}
\end{definition}

\begin{lemma}\label{hhhhh}
    The Legendre transform is a diffeomorphism as a map from $\op{int}(\C)$ to $T^*M\backslash \C^*$.
\end{lemma}
\begin{proof}
   Since the second derivative along the fibers of $L$ is positive definite, the Legendre transform is injective and a diffeomorphism onto its image, thanks to the inverse function theorem. Thus, the only remaining task is to verify that the image of ${\cal L}$ is precisely the complement of the closed dual future cone. If $v\in \op{int}(\C_x)$, we have
    \begin{align*}
      \frac{\partial L}{\partial v}(x,v)(v)=2 L(x,v)>0. 
    \end{align*}
   By the characterisation of $\C_x$ given in Remark \ref{dddd}, it follows that $ {\cal L}(x,v)\in T^*M\backslash \C^*$. Conversely, let $(x,p)\in T^*M\backslash \C^*$. Then there is $v_0\in \C_x$ with $pv_0>0$. Then we have
   \begin{align}
       p(\lambda v_0)-L(x,\lambda v_0)= \lambda pv_0-\lambda^2 L(x,v_0)>0 \text{ for small } \lambda>0. \label{dasdadsa}
   \end{align}
   Thus,
   \begin{align*}
       H(x,p):=\sup_{v\in \C_x} (pv-L(x,v))>0. 
   \end{align*}
   Moreover, using the notation of Definition \ref{e1}, we have
   \begin{align*}
       pv-L(x,v)\leq |p|_h |v|_h-(|p|_h+1)|v|_h +C(|p|_h+1) \xrightarrow{|v|_h\to \infty} -\infty.
   \end{align*}
   Therefore the continuity of the map $\C_x\to \R,\ v\mapsto pv-L(x,v),$ shows that there is $v_1\in \C_x$ such that
   \begin{align*}
       0<H(x,p)= pv_1-L(x,v_1)\leq pv_1
   \end{align*}
   In particular we must have $pv_1>0$ and $v_1\neq 0$. We show that $v_1\in \op{int}(\C_x)$. Indeed, if this was not the case then $v_1\in \partial\C_x\backslash \{0\}$. Then, working in an orthonormal basis of $T_xM$, it is easy to check that there is $v_2$ close to $v_1$ with
   \begin{align*}
       pv_2-L(x,v_2)>pv_1-L(x,v_1).
   \end{align*}
   So, $v_1\in \op{int}(\C_x)$ and since $L(x,\cdot)$ is smooth on the open future cone, we obtain by differentiation
   \begin{align*}
       0=p-\frac{\partial L}{\partial v}(x,v_1) \Rightarrow p={\cal L}(x,v_1).
   \end{align*}
   This ends the proof of the lemma.
\end{proof}

\begin{definition}\rm
The \emph{Hamiltonian} is the map
\[
H:T^*M\to \R,\ H(x,p):=\sup_{v\in T_xM} pv-L(x,v).
\]
\end{definition}

\begin{remark}\rm \label{lllll}
The proof of the above lemma shows that, for $(x,v)\in \op{int}(\C)$, we have
\[
H(x,{\cal L}(x,v))={\cal L}(x,v)(v)-L(x,v).
\]
Since the Legendre transform ${\cal L}:\op{int}(\C)\to T^*M\backslash \C^*$ is a diffeomorphism, it follows that $H$ is smooth on $T^*M\backslash \C^*$. Moreover, it is well-known that the Euler-Lagrange flow $\phi_t$ on $\op{int}(\C)$ is conjugated, via ${\cal L}$, to the Hamiltonian flow $\psi^H_t$ on $T^*M\backslash \C^*$. Here, $T^*M\backslash \C^*$ is equipped with its natural symplectic structure.
\end{remark}

\section{Lax-Oleinik semigroup}
\begin{definition}\rm
The \emph{forward Lax-Oleinik semigroup} is the family of maps $(T_t)_{ t\geq 0}$, defined on the space of functions $u:M\to \overbar \R$ by
\begin{align*}
    T_tu:M\to \overbar \R,\ T_tu(x):=\inf\{u(y)+c_t(y,x)\mid y\in M\},
\end{align*}
where we use the convention $-\infty+\infty:=+\infty$.
The \emph{backward Lax-Oleinik semigroup} is the family of maps $(\hat T_t)_{t\geq 0}$, defined on the space of functions $u:M\to \overbar \R$ by
\begin{align*}
    \hat T_tu:M\to \overbar \R,\ \hat T_tu(x):=\sup\{u(y)-c_t(x,y)\mid y\in M\},
\end{align*}
where we use the convention $\infty-\infty:=-\infty$.
If $u:M\to \overbar \R$ is any function, we also write
\begin{align*}
    Tu:&\ [0,\infty)\times M\to \overbar \R, \ (t,x)\mapsto T_tu(x),
\end{align*}
and we use the analogous notation $\hat Tu$.
\end{definition}

\begin{remark}\rm
Note that, according to our convention,
\begin{align*}
    &T_tu(x)=\inf\{u(y)+c_t(y,x)\mid y\in J^-(x)\} \text{ and } 
    \\
    &\hat T_tu(x)=\sup\{u(y)-c_t(x,y)\mid y\in J^+(x)\}
\end{align*}
This eliminates the need to specify a convention for the sum $\mp\infty\pm\infty$.
\end{remark}

\begin{lemma} \label{y}
    The Lax-Oleinik semigroups indeed form a semigroup, meaning that for $u:M\to \overbar \R$ and $t,s\geq 0$, we have the identities $T_{t+s}u=T_t(T_su)$ and $\hat T_{t+s}u=\hat T_t(\hat T_su)$.
\end{lemma}
\begin{proof}
    This follows as in \cite{Fathi}, where the result is proven for Tonelli Lagrangians on a compact manifold. The only property required in the proof is that, for $x,y\in M$ and $t,s\geq 0$, we have
    \begin{align*}
        c_{t+s}(y,x)=\inf_{z\in M} c_t(y,z)+c_s(z,x).
    \end{align*}
    The inequality \q{$\leq$} follows directly from the definition of the minimal action. The reverse inequality holds because, for any pair $(y,x)\in M\times M$, we can find a minimizer $\gamma:[0,t+s]\to M$ that connects $y$ to $x$. We then set $z:=\gamma(t)$.
\end{proof}

The following lemma is well-known but we provide a proof since we have not encountered this exact formulation in a work where the minimal action can attain the value $+\infty$. A very similar situation is presented in \cite{Ambrosio}, Definition 6.1.2.

\begin{lemma}\label{ineq}
    For $u:M\to \overbar \R$ and $t>0$, we have $\hat T_t(T_tu)\leq u$ and also $T_t(\hat T_tu)\geq u$.
\end{lemma}
\begin{proof}
    Let $x\in M$. We have
    \[
    \hat T_t(T_tu)(x)=\sup_{y\in J^+(x)} \inf_{z\in J^-(y)} u(z)+c_t(z,y)-c_t(x,y)
    \]
    Choosing $z=x$, we obtain $\hat T_t(T_tu)(x)\leq u(x)$. The argument for the second inequality is similar.
\end{proof}

\begin{remark}\rm \label{vv}
 For Tonelli Lagranians on a compact manifold $N$, it is well known that $Tu$ is locally semiconcave for any continuous function $u$, as stated in \cite{Fathi}. This result can be extended to the case of Lipschitz functions on non-compact manifolds, see the proof of Theorem 3.3 (4) in \cite{FathiWK}. We will extend this result (under additional assumptions) to the Lorentzian case in the next section. Even in the case for Tonelli Lagrangians on a non-compact manifold and arbitrary functions $u:N\to \overbar \R$, it was shown in \cite{FathiHJ}, Theorem 6.2(iii), that if $Tu$ is finite at some point $(T,X)\in (0,\infty)\times N$, then $Tu$ is locally semiconcave on $(0,T)\times N$. 
\medskip

However, such a result cannot be generalized to the Lorentzian case: Consider a function $u:M\to \overbar \R$ and two points $x_0,y_0\in M$ such that $u(x_0)$ and $T_1u(y_0)$ are finite. The only information we have about $Tu$ is on the set $(0,1)\times (J^+(x_0)\cap J^-(y_0)$. Nonetheless, one might still hope that $u$ is at least locally semiconcave on the interior of this set. 

Our main result in this section asserts that the Lax-Oleinik evolution $Tu$ is locally semiconcave on an open neighbourhood of the relative interior of the graph of calibrated curves. For the precise formulation, see below. Let us note that our proofs utilize several ideas from \cite{FathiHJ}, Theorem 6.2, and we have adapted both the notation and the formulation of the results. 
\end{remark}

\begin{theorem}\label{m}
\begin{enumerate}[(i)]
    \item Let $u:M\to \overbar \R$ be any function, and let $y_0\in I^+(x_0)$ and $t_0>0$. Suppose that $u(x_0)$ and $T_{t_0}u(y_0)$ are finite. Additionally, assume that 
    $T_{t_0}u(y_0)=u(x_0)+c_{t_0}(x_0,y_0)$. Let $\gamma:[0,t_0]\to M$ be a minimizer connecting $x_0$ with $y_0$ and let $t_1\in (0,t_0)$. Then the Lax-Oleinik semigroup evolution $Tu$ is locally semiconcave on a neighborhood of $(t_1,\gamma(t_1))$.

      \item Let $u:M\to \overbar \R$ be any function, and let $y_0\in I^-(x_0)$ and $t_0>0$. Suppose that $u(x_0)$ and $\hat T_{t_0}u(y_0)$ are finite. Additonally, assume that 
    $\hat T_{t_0}u(y_0)=u(x_0)-c_{t_0}(y_0,x_0)$. Let $\gamma:[0,t_0]\to M$ be a minimizer connecting $y_0$ with $x_0$ and let $t_1\in (0,t_0)$. Then the Lax-Oleinik semigroup evolution $\hat Tu$ is locally semiconvex on a neighborhood of $(t_1,\gamma(t_0-t_1))$.
    \end{enumerate}
\end{theorem}

\begin{remark}\rm 
\begin{enumerate}[(a)]
\item
We will prove only the first part of the theorem, as the proofs follow a similar structure. To proceed, we require some preparatory lemmas, which we will state (almost) exclusively for the forward Lax-Oleinik semigroup. However, it is important to note that each of these lemmas has a corresponding version for the backward Lax-Oleinik semigroup.
\item Except for minor modifications, the next two Lemmas are borrowed from \cite{FathiHJ}.  
\end{enumerate}
\end{remark}

\begin{lemma}[\cite{FathiHJ}, Theorem 6.2] \label{ww}
    Let $u:M\to \overbar \R$ be any function, and let $x_0,y_0\in M$ and $t_0>0$. Suppose that $u(x_0)$ and $T_{t_0}u(y_0)$ are finite. Then the function $Tu$ is locally bounded on the set $(0,t_0)\times (J^+(x_0)\cap J^-(y_0)).$
\end{lemma}
\begin{proof}
    Let $K\subseteq (0,t_0)\times (J^+(x_0)\cap J^-(y_0))$ be a compact set. The definition of the Lax-Oleinik semigroup yields
    \begin{align*}
        T_tu(x)\leq u(x_0)+c_t(x_0,x),\ (t,x)\in (0,\infty)\times M.
    \end{align*}
    By continuity of the mapping $(0,\infty)\times J^+,\ (t,x,y)\mapsto c_t(x,y)$, and compactness of $K$ it follows that $Tu$ is bounded above on $K$ by
    \begin{align}
         u(x_0)+\sup_{(t,x)\in K} c_t(x_0,x). \label{qwert1}
    \end{align}
    On the other hand, by the same argument, we have for $(t,x)\in K$
    \begin{align}
        T_{t_0}u(y_0)\leq T_tu(x)+c_{1-t}(x,y_0)\leq T_tu(x)+\sup_{(\tilde t,\tilde x)\in K} c_{1-\tilde t}(\tilde x,y_0). \label{qwert}
    \end{align}
    Combining the inequalities \eqref{qwert1} and \eqref{qwert}, we arrive at the following bounds for $T_tu(x)$, $(t,x)\in K$:
    \begin{align*}
       T_{t_0}u(y_0)- \sup_{(\tilde t,\tilde x)\in K} c_{1-\tilde t}(\tilde x,y_0)\leq  T_tu(x) \leq  u(x_0)+\sup_{(\tilde t,\tilde x)\in K} c_{\tilde t}(x_0,\tilde x)
    \end{align*}
    Since, by assumption, $u(x_0)$ and $T_{t_0}u(y_0)$ are finite, the lemma follows.
\end{proof}

\begin{lemma}[\cite{FathiHJ}, Proposition 7.5] \label{w}
    Let $u:M\to \overbar \R$ be any function, and let $y_0\in J^+(x_0)$ and $t_0>0$. Suppose that $u(x_0)$ and $T_{t_0}u(y_0)$ are finite. Additionally, assume that 
    $T_{t_0}u(y_0)=u(x_0)+c_{t_0}(x_0,y_0)$. Let $\gamma:[0,t_0]\to M$ be a minimizer connecting $x_0$ with $y_0$.
    Then
    \begin{align*}
 T_tu(\gamma(t))=T_su(\gamma(s))+c_{t-s}(\gamma(s),\gamma(t)) \text{ for any } 0\leq s \leq t \leq t_0.
    \end{align*}
    We shall say that $\gamma$ is \emph{$Tu$-calibrated}.
\end{lemma}
\begin{proof}
    Using Lemma \ref{y} we obtain the following inequalities
    \begin{align}
        &T_su(\gamma(s))\leq u(x_0)+c_t(x_0,\gamma(s)), \nonumber
        \\[10pt]
        & T_tu(\gamma(t))\leq T_su(\gamma(s))+c_{t-s}(\gamma(s),\gamma(t)), \label{x}
        \\[10pt]
        & T_{t_0}u(y_0)\leq T_tu(\gamma(t))+c_{t_0-t}(\gamma(t),y_0). \nonumber
    \end{align}
    Adding these three inequalities and canceling terms (noting that all involved quantities are finite), we get
    \begin{align*}
        T_{t_0}u(y_0)&\leq u(x_0)+c_s(x_0,\gamma(s))+c_{t-s}(\gamma(s),\gamma(t))+c_{t_0-t}(\gamma(t),y_0)
        \\[10pt]
        &=u(x_0)+c_{t_0}(x_0,y_0),
    \end{align*}
    where the last inequality follows from the fact that $\gamma$ is a minimizer,
    By assumption, this inequality must actually hold as an equality, which implies that equality must also hold in \eqref{x} (and also in the other two inequalities). This completes the proof of the lemma.
\end{proof}

\begin{lemma} \label{ddd}
\begin{enumerate}[(i)]
    \item 
   Let $u:M\to \overbar \R$ be any function, and let $y_0\in I^+(x_0)$ and $t_0>0$. Suppose that $u(x_0)$ and $T_{t_0}u(y_0)$ are finite. Additionally, assume that 
    $T_{t_0}u(y_0)=u(x_0)+c_{t_0}(x_0,y_0)$.
    Then the function $Tu$ is super-differentiable at $(t_0,y_0)$, and
    \begin{align}
    (\partial_t c_{t_0}(x_0,y_0),\frac{\partial L}{\partial v}(y_0,\dot \gamma(t_0))) \in \partial^+ Tu(t_0,y_0) \label{zz}
    \end{align}
    where $\gamma:[0,t_0]\to M$ is a (necessarily timelike) minimizer connecting $x_0$ to $y_0$ in time $t_0$.
    Moreover, 
    \begin{align}
    \frac{\partial L}{\partial v}(y_0,\dot \gamma(0))\in \partial^- u(x_0). \label{zzz}
    \end{align}

    \item Let $u:M\to \overbar \R$ be any function, and let $y_0\in I^-(x_0)$ and $t_0>0$. Suppose that $u(x_0)$ and $\hat T_{t_0}u(y_0)$ are finite. Additionally, assume that 
    $\hat T_{t_0}u(y_0)=u(x_0)-c_{t_0}(y_0,x_0)$.
    Then the function $\hat Tu$ is sub-differentiable at $(t_0,y_0)$, and
    \[
    (-\partial_t c_{t_0}(y_0,x_0),\frac{\partial L}{\partial v}(y_0,\dot \gamma(0)))\in \partial^- Tu(t_0,y_0)
    \]
    where $\gamma:[0,t_0]\to M$ is a (necessarily timelike) minimizer connecting $y_0$ to $x_0$ in time $t_0$.
     Moreover,
    \[
    \frac{\partial L}{\partial v}(x_0,\dot \gamma(t_0))\in \partial^+ u(x_0).
    \]
    \end{enumerate}
\end{lemma}
\begin{proof}
Let us only prove (i), as (ii) follows by a similar argument. Since $y_0\in I^+(x_0)$, it is clear that any minimizer connecting $x_0$ to $y_0$ must be timelike. Let us prove \eqref{zz}.

    We have $T_tu(y)\leq u(x_0)+c_t(x_0,y)$ for all $(t,y)$, with equality at $(t,y)=(t_0,y_0)$. Since $y_0\in I^+(x_0)$, Theorem \ref{A} ensures that the map $(t,y)\mapsto c_t(x_0,y)$ is super-differentiable at $(t_0,y_0)$, with a super-differential given by the vector in \eqref{zz}. Using these two facts, it is now easy to prove \eqref{zz}. 
    
    To establish \eqref{zzz}, we apply the same reasoning. Indeed, we observe that $u(x)\geq T_{t_0}u(y_0)-c_{t_0}(x,y_0)$ with equality at $x=x_0$. Since $x\mapsto -c_{t_0}(x,y)$ is sub-differentiable at $x_0$, with a sub-differential given by the vector in \eqref{zzz}, this yields the second statement.  
\end{proof}

\begin{corollary}\label{xx}
       Let $u:M\to \overbar \R$ be any function, and let $y_0\in I^+(x_0)$ and $t_0>0$. Suppose that $u(x_0)$ and $T_{t_0}u(y_0)$ are finite. Additionally, assume that 
    $T_{t_0}u(y_0)=u(x_0)+c_{t_0}(x_0,y_0)$. Let $\gamma:[0,t_0]\to M$ be a minimizer connecting $x_0$ with $y_0$. 
Then the function $Tu$ is differentiable at $(t_1,\gamma(t_1))$ for any $t_1\in (0,t_0)$.
\end{corollary}
\begin{proof}
    From Lemma \ref{w} we know $T_{t_1}u(\gamma(t_1))=u(x_0)+c_{t_1}(x_0,\gamma(t_1))$. Thus, by part (i) of the above lemma, the Lax-Oleinik evolution $Tu$ is super-differentiable at $(t_1,\gamma(t_1))$.  Moreover, by Lemma \ref{y} and Lemma \ref{ineq}, we have for all $t\in (0,t_0)$ and all $x\in M$ that
    \begin{align*}
    \hat T_{t_0-t}T_{t_0}u(x)=\hat T_{t_0-t}T_{t_0-t}T_{t}u(x)\leq T_{t}u(x),
    \end{align*}
    and, utilizing the fact that $\gamma$ is a minimizer, it is straightforward to verify that equality holds at $(t_1,\gamma(t_1))$. Moreover, one checks that
    \[ \hat T_{t_0-t_1}T_{t_0}u(\gamma(t_1))=T_{t_0}u(y_0)-c_{t_0-t_1}(\gamma(t_1),\gamma(t_0)).\]
    Thus, by (ii) of the above lemma, the mapping 
    \[
    (t,x)\mapsto (t_0-t,x)\mapsto \hat T_{t_0-t}T_{t_0}u(x)
    \]
    is sub-differentiable at $(t,x)=(t_1,\gamma(t_1))$ as the composition of a sub-differentiable function with a smooth function. These facts easily imply that also $Tu$ is sub-differentiable at $(t_1,\gamma(t_1))$ with the same sub-differential.

    Since $Tu$ is both super- and sub-differentiable at $(t_1,\gamma(t_1))$, it follows that $Tu$ is differentiable at that point.
\end{proof}

\begin{remark}\rm
The results presented so far are are well-known in the case for Tonelli-Lagrangians, and the proofs are essentially the same. Of course, there is no counterpart for the assumption $y_0\in I^+(x_0)$. The next Lemma is borrowed from \cite{FathiHJ}, Theorem 6.2, where the result is stated for Tonelli-Lagrangians. The proof is essentially the same, but we provide it here again for the reader's convenience.
\end{remark}

\begin{lemma}
Let $u:M\to \overbar \R$ be any function, and let $y_0\in I^+(x_0)$ and $t_0>0$. Suppose that $u(x_0)$ and $T_{t_0}u(y_0)$ are finite. Suppose that $\delta>0$ and $(t_1,x_1)\in \R \times M$ are such that
\begin{align*}
    [t_1-\delta,t_1+\delta]\times \overbar B_{\delta}(x_1)\subseteq (0,t_0)\times (J^+(x_0)\cap J^-(y_0)).
\end{align*}
Then there exists $\ep>0$ with $2\ep<\delta$, such that, for all $(t,x)\in [t_1-\ep,t_1+\ep]\times \overbar B_{\ep}(x_1)$, we have
\begin{align*}
    T_tu(x)=\inf\{T_su(y)+c_{t-s}(y,x)\mid (s,y)\in [t_1-\delta,t_1-2\ep]\times \overbar B_{\delta}(x_1)\}.
\end{align*}
\end{lemma}
\begin{notation}\rm\label{notation}
    Throughout this paper, we denote by $B_r(x)$ the open ball of radius $r$ centered at $x$, either w.r.t.\ the complete metric $h$ on $M$ (when $x\in M$), or w.r.t.\  the Euclidean metric on $\R^n$ (when $x\in \R^n$).
\end{notation}
\begin{proof}
    Denote $K:=[t_1-\delta,t_1+\delta]\times \overbar B_{\delta}(x_1)$ and choose $\delta/2>\ep>0$ such that
    \begin{align}
        \frac{(\delta-\ep)^2}{12\ep}> 2\sup_{(t,x)\in K} |T_tu(x)|+1. \label{szuaisjaiooa}
    \end{align}
    The finiteness of the supremum on the right-hand side follows from Lemma \ref{ww}.
    Let $(t,x)\in [t_1-\ep,t_1+\ep]\cap \overbar B_\ep( x_1)$ be given, and let $0<\eta<1$ be arbitrary. We can find $y\in M$ such that
    \begin{align*}
        u(y)+c_t(y,x)\leq T_tu(x)+\eta.
    \end{align*}
    Now, let $\gamma:[0,t]\to M$ be a minimizer with $\gamma(0)=y$ and $\gamma(t)=x$.
Since $\gamma$ is a minimizer, we have for $s\in [0,t]$
\begin{align*}
    c_t(\gamma(0),x)=c_s(\gamma(0),\gamma(s))+c_{t-s}(\gamma(s),x).
\end{align*}
Substituting this into the inequality above gives
\begin{align*}
    u(y)+c_s(y,\gamma(s))+c_{t-s}(\gamma(s),x)\leq T_tu(x)+\eta.
\end{align*}
Using the fact that, by definition of $T_su(\gamma(s))$, we have $T_su(\gamma(s))\leq u(y)+c_s(y,\gamma(s))$, we obtain 
\begin{align}
    T_su(\gamma(s))+c_{t-s}(\gamma(s),x)\leq T_tu(x)+\eta. \label{qwehj}
\end{align}
Now, we pick $s$ such that $(s,\gamma(s))\in \partial([t_1-\delta,t_1]\times \overbar B_{\delta}(x_1))$. This is possible since $(0,\gamma(0))\notin [t_1-\delta,t_1]\times \overbar B_{\delta}(x_1)$ (because $t_1-\delta>0$), and $(t,\gamma(t))=(t,x)\in [t_1-\ep,t_1+\ep]\times \overbar B_{\ep}(x_1)$). It follows that either $s=t_1-\delta$ or $d_h(\gamma(s),x_1)=\delta$. We want to show that in both cases $s\leq t_1-2\ep$.

In the first case, this holds since $\ep<\delta/2$.

Now, let us turn to the second case, where $d_h(\gamma(s),x_1)=\delta$. The triangle inequality yields
\begin{align*}
    d_h(\gamma(s),x)\geq d_h(\gamma(s),x_1)-d_h(x_1,x)\geq \delta-\ep.
\end{align*}
Due to \eqref{splitting}, the following inequality holds:
\begin{align*}
    c_{t-s}(\gamma(s),x)\geq \frac{d_h(\gamma(s),x)^2}{4(t-s)} \geq \frac{(\delta-\ep)^2}{4(t-s)}.
\end{align*}
Inserting this into \eqref{qwehj} gives
\begin{align*}
     T_su(\gamma(s))+\frac{(\delta-\ep)^2}{4(t-s)}\leq T_tu(x)+\eta.
\end{align*}
Comparing this with \eqref{szuaisjaiooa} shows that $t-s\geq 3\ep$, so that $t_1-s\geq 2\ep$. Thus, in both cases, we have $s\leq t_1-2\ep$. 

In particular, we have proven that
\begin{align*}
    \inf\{T_su(y)+c_{t-s}(y,x)\mid (s,y)\in [t_1-\delta,t_1-2\ep]\times \overbar B_{\delta}(x_1)\}\leq T_tu(x)+\eta.
\end{align*}
Since $\eta$ was arbitrary, the proof of the lemma is complete.
\end{proof}

\begin{remark}\rm \label{yy}
The preceding lemma shows that, in the case of a Tonelli-Lagrangian system, using the same notation as above (with $h_t$ replacing $c_t$ for the minimal action), the Lax-Oleinik evolution of $u$ is locally semiconcave on the set $(t_1-\ep,t_1+\ep)\times B_\ep(x_1)$. The reason is that the family of functions $(t,x)\mapsto h_{t-s}(x,y)$, as $(s,y)$ varies over the compact set $[t_1-\delta,t_1-2\ep]\times \overbar B_\delta(x_1)$, is uniformly locally semiconcave on this set in the sense of \cite{Fathi/Figalli}, Definition A15. Consequently, the infimum (which, by the lemma, coincides with  $T_tu(x)$) is indeed locally semiconcave (\cite{Fathi/Figalli}, Proposition A16).

In the Lorentzian case, however, the difficulty arises from the fact that the family of mappings $(t_1-\ep,t_1+\ep)\times B_\ep(x_1)\to \overbar \R,\ (t,x)\mapsto c_{t-s}(x,y)$, as $(s,y)$ varies over the compact set $[t_1-\delta,t_1-2\ep]\times \overbar B_\delta(x_1)$, is not uniformly locally semiconcave. In fact, these functions are not even continuous or everywhere finite. We only know that $c_t(x,y)$ is locally semiconcave on $(0,\infty)\times I^+$. 

Thus, in order to obtain uniform semiconcavity results, we have to show that, locally for $x$, the infimum in the definition of $T_tu(x)$ can be taken over all $y$ that remain uniformly bounded away from $\partial J^-(x)$. The first step into this direction is to show that $y$ can be chosen uniformly bounded away from $x$. This is the content of the following lemma.
\end{remark}

\begin{lemma}
 Let $u:M\to \overbar \R$ be any function, and let $y_0\in I^+(x_0)$ and $t_0>0$. Suppose that $u(x_0)$ and $T_{t_0}u(y_0)$ are finite. Additionally, assume that 
    $T_{t_0}u(y_0)=u(x_0)+c_{t_0}(x_0,y_0)$. Let $\gamma:[0,t_0]\to M$ be a minimizer connecting $x_0$ with $y_0$,
     and let $t_1\in (0,t_0)$. Suppose that $\delta>0$ is such that
\begin{align*}
    [t_1-\delta,t_1+\delta]\times \overbar B_{\delta}(\gamma(t_1))\subseteq (0,t_0)\times (J^+(x_0)\cap J^-(y_0)).
\end{align*}
Then there exists $\ep>0$ with $2\ep<\delta$, such that, for all $(t,x)\in [t_1-\ep,t_1+\ep]\times \overbar B_{\ep}(\gamma(t_1))$, we have
\begin{align*}
    T_tu(x)=\inf\{&T_su(y)+c_{t-s}(y,x)\mid 
    \\
    &(s,y)\in [t_1-\delta,t_1-2\ep]\times (\overbar B_{\delta}(\gamma(t_1))\backslash  B_\ep(\gamma(t_1)))\}.
\end{align*}
\end{lemma}
\begin{proof}
First, let $\ep_1>0$ be given by the above lemma. Then we know that, for all $(t,x)\in [t_1-\ep_1,t_1+\ep_1]\times \overbar B_{\ep_1}(\gamma(t_1))$, we have
\begin{align}
    T_tu(x)=\inf\{&T_su(y)+c_{t-s}(y,x)\mid \nonumber
    \\
    &(s,y)\in [t_1-\delta,t_1-2\ep_1]\times \overbar B_{\delta}(\gamma(t_1))\}. \label{plkjhgfd}
\end{align}
Clearly, it suffices to prove that there exists $\ep<\ep_1$ such that, for $(t,x)\in [t_1-\ep,t_1+\ep]\times \overbar B_{\ep}(\gamma(t_1))$, we have
\begin{align}
        T_tu(x)<\inf\{T_su(y)\mid (s,y) \in [t_1-\delta,t_1-2\ep_1]\times   B_\ep(\gamma(t_1))\}. \label{a1}
 \end{align}
By Lemma \ref{w}, $\gamma$ is $Tu$-calibrated, so
    \begin{align}
        T_{t_0}u(y_0)=T_{t_1}u(\gamma(t_1))+c_{t_0-t_1}(\gamma(t_1),y_0). \label{v}
    \end{align}
    Now $Tu$ is continuous at $(t_1,\gamma(t_1))$ thanks to Corollary \ref{xx}. Moreover, since $y_0\in I^+(\gamma(t_1))$, we also have $y_0\in I^+(x)$ if $x$ is sufficiently close to $\gamma(t_1)$, implying that $c_{t_0-t}(x,y_0)\to c_{t_0-t_1}(\gamma(t_1),y_0)$ as $(t,x)\to (t_1,\gamma(t_1))$. Combining these facts with \eqref{v}, we obtain
    \begin{align*}
      T_{t}u(x)+c_{t_0-t}(x,y_0) \to T_{t_1}u(\gamma(t_1))+c_{t_0-t_1}(\gamma(t_1),y_0)=T_{t_0}u(y_0)
    \end{align*}
    as $(t,x)\to (t_1,\gamma(t_1))$. In particular, there exists a modulus of continuity $\omega$ such that
    \begin{align}
        T_{t_0}u(y_0)\geq T_{t}u(x)+c_{t_0-t}(x,y_0) - \omega(|t-t_1|+d_h(x,\gamma(t_1))). \label{dosifsk}
    \end{align}
    Additionally, we also have the following estimate:
    \begin{align*}
        T_{t_0}u(y_0)\leq T_su(y)+c_{t_0-s}(y,y_0)
    \end{align*}
    for all $y\in M$ and $s\in [0,t_0)$. 
    Comparing this inequality with \eqref{dosifsk} we have shown that
    \begin{align}
        T_{t}u(x)& \leq  T_su(y) +c_{t_0-s}(y,y_0) -c_{t_0-t}(x,y_0) \nonumber
        \\[10pt]
        & +\omega(|t-t_1|+d_h(x,\gamma(t_1))) \label{a2}
    \end{align}
    Using the formula $c_\tau=c/\tau$ for $\tau>0$ and the continuity of $c$ at $(\gamma(t_1),y_0)\in I^+$, it is easy to check that
    \begin{align*}
       \limsup_{\substack{x,y\to \gamma(t_1)\\ t\to t_1}} \sup_{s\leq t_1-2\ep_1} \bigg(c_{t_0-s}(y,y_0) -c_{t_0-t}(x,y_0) + \omega(|t-t_1|+d_h(x,\gamma(t_1)))\bigg)<0.
    \end{align*}
    This, together with \eqref{a2}, implies that there exists $\ep<\ep_1$ such that \eqref{a1} holds.
\end{proof}

\begin{lemma}
 Let $u:M\to \overbar \R$ be any function, and let $y_0\in I^+(x_0)$ and $t_0>0$. Suppose that $u(x_0)$ and $T_{t_0}u(y_0)$ are finite. Additionally, assume that 
    $T_{t_0}u(y_0)=u(x_0)+c_{t_0}(x_0,y_0)$. Let $\gamma:[0,t_0]\to M$ be a minimizer connecting $x_0$ with $y_0$,
     and let $t_1\in (0,t_0)$.
     Suppose that $\delta>0$ is such that
\begin{align*}
    [t_1-\delta,t_1+\delta]\times \overbar B_{\delta}(\gamma(t_1))\subseteq (0,t_0)\times (J^+(x_0)\cap J^-(y_0)).
\end{align*}
Then there exists $\ep>0$ with $2\ep<\delta$, such that, for all $(t,x)\in [t_1-\ep,t_1+\ep]\times \overbar B_{\ep}(\gamma(t_1))$, we have
\begin{align}
    T_tu(x)=\inf\{&T_su(y)+c_{t-s}(y,x)\mid \nonumber
    \\
    &(s,y)\in [t_1-\delta,t_1-2\ep]\times \overbar B_{\delta}(\gamma(t_1)),\ d(y,\gamma(t_1))\geq \ep\}. \label{a}
\end{align}
\end{lemma}
\begin{proof}
By the foregoing lemma, we know that there exists $\ep_1<\delta/2$ such that,  
for all $(t,x)\in [t_1-\ep_1,t_1+\ep_1]\times \overbar B_{\ep_1}(\gamma(t_1))$, we have
\begin{align*}
    T_tu(x)=\inf\{&T_su(y)+c_{t-s}(y,x)\mid 
    \\
    &(s,y)\in [t_1-\delta,t_1-2\ep_1]\times (\overbar B_{\delta}(\gamma(t_1))\backslash  B_{\ep_1}(\gamma(t_1)))\}.
\end{align*}
Thus, it suffices to prove that there exists $\ep<\ep_1$ such that, for $(t,x)\in [t_1-\ep,t_1+\ep]\times \overbar B_{\ep}(\gamma(t_1))$, we have
\begin{align*}
   T_tu(x)<&\inf\{T_su(y)+c_{t-s}(y,x)\mid 
   \\
   &(s,y)\in [t_1-\delta,t_1-2\ep_1]\times  (\overbar B_{\delta}(\gamma(t_1))\backslash  B_{\ep_1}(\gamma(t_1))),\ d(y,\gamma(t_1))< \ep\}. 
\end{align*}
We argue by contradiction and assume that we find a sequence $(t_k,x_k)\to (t_1,\gamma(t_1))$ and, for each $k$, a pair 
    \begin{align*}
        (s_k,y_k)\in [t_0-\delta,t_0-2\ep_1]\times (\overbar B_{\delta}(\gamma(t_1))\backslash  B_{\ep_1}(\gamma(t_1))) \text{ such that } d(y_k,x_k)\to 0 
    \end{align*}
    and such that
    \begin{align}
        T_{s_k}u(y_k) +c_{t_k-s_k}(y_k,x_k)-\frac{1}{k}\leq T_{t_k}u(x_k). \label{ww'}
    \end{align}
		Of course we must have $y_k\in J^-(x_k)$ for large $k$.
 Since $Tu$ is continuous at $(t_1,\gamma(t_1))$, the latter converges to $T_{t_1}u(\gamma(t_1))$ as $k\to \infty$.
    Since $\gamma$ is $Tu$-calibrated, we have
    \begin{align}
        T_{t_0}u(y_0)=T_{t_1}u(\gamma(t_1))+c_{t_0-t_1}(\gamma(t_1),y_0).\label{poiu1}
    \end{align}
    On the other hand, by the semigroup property of $Tu$ it also holds that
    \begin{align}
        T_{t_0}u(y_0)\leq T_{s_k}u(y_k)+c_{t_0-s_k}(y_k,y_0).\label{poiu2}
    \end{align}
    Now, from \eqref{ww'},\eqref{poiu1} and \eqref{poiu2} we obtain
    \begin{align*}
        T_{s_k}u(y_k)+c_{t_k-s_k}(y_k,x_k)-\frac 1k
        \leq & T_{t_k}u(x_k)-T_{t_1}u(\gamma(t_1)+T_{t_1}u(\gamma(t_1)
        \\
        \leq &
        T_{t_k}u(x_k)-T_{t_1}u(\gamma(t_1)+ T_{s_k}u(y_k)
        \\
        &+c_{t_0-s_k}(y_k,y_0)-c_{t_0-t_1}(\gamma(t_1),y_0).
    \end{align*}
   Rearranging this inequality leads to
    \begin{align*}
        c_{t_k-s_k}(y_k,x_k)+c_{t_0-t_1}(\gamma(t_1),y_0) \leq T_{t_k}u(x_k)-T_{t_1}u(\gamma(t_1)+c_{t_0-s_k}(y_k,y_0).
    \end{align*}
    Since the metric $h$ is complete and due to the construction of our sequence, $(s_k,y_k)$ is precompact. Thus, we can assume that, along a subsequence that we do not relabel, $(s_k,y_k)\to (s_1,y_1)\in [t_1-\delta,t_1-2\ep_1]\times (\overbar B_\delta(\gamma(t_1))\backslash B_{\ep_1}(\gamma(t_1)))$.
    Since $J^+$ is closed, it follows $y_1\in J^-(\gamma(t_1))$, and by the continuity of the Lorentzian distance, $d(y_1,\gamma(t_1))=0$. Taking limits in the above inequality and regarding the continuity of $Tu$ at $(t_1,\gamma(t_1))$, we conclude
    \begin{align*}
        c_{t_1-s_1}(y_1,\gamma(t_1))+c_{t_0-t_1}(\gamma(t_1),y_0)\leq c_{t_0-s_1}(y_1,y_0).
    \end{align*}
    Since the reverse inequality always holds, this must be an equality. This implies that the curve obtained by concatenating a minimizer from $y_1$ to $\gamma(t_1)$ with the curve $\gamma_{|[t_1,t_0]}$ is still a minimizer. However, the first curve is not timelike, since $d(y_1,\gamma(t_1))=0$, and non-trvial ($y_1\neq \gamma(t_1)$ (this is where we need the last lemma)). The second curve, on the other hand, is timelike because $(\gamma(t_1),y_0)\in I^+$. This contradicts Lemma \ref{minimizer}, completing the proof.
\end{proof}

\begin{remark}\rm
    We are now ready to prove Theorem 3.6.
\end{remark}

\begin{proof}[Proof of Theorem \ref{m}]
    Let $\delta>0$ be such that 
    \begin{align*}
    [t_1-\delta,t_1+\delta]\times \overbar B_{\delta}(\gamma(t_1))\subseteq (0,t_0)\times (J^+(x_0)\cap J^-(y_0)).
\end{align*}
We can then apply the preceding lemma. Thus, there exists $\ep<\delta/2$ such that, for all $(t,x)\in [t_1-\ep,t_1+\ep]\times \overbar B_{\ep}(\gamma(t_1))$, we have
\begin{align*}
    T_tu(x)=\inf\{&T_su(y)+c_{t-s}(y,x)\mid
    \\
    &(s,y)\in [t_1-\delta,t_1-2\ep]\times \overbar B_{\delta}(\gamma(t_1)),\ d(y,\gamma(t_1))\geq \ep\}. 
\end{align*}  
Due to the continuity of the Lorenzian distance, it readily follows that there exists $\ep>r>0$ and a compact set $K$ such that $B_r(\gamma(t_1))\times K\subseteq I^+$ and such that, for all $(t,x)\in [t_1-\ep,t_1+\ep]\times B_r(\gamma(t_1))$, we have
\begin{align}
T_tu(x)=\inf\{T_su(y)+c_{t-s}(y,x)\mid s\in [t_1-\delta,t_1-2\ep],\ y\in K\}. \label{c1}
\end{align}
Since $\C$ is locally semiconcave on $(0,\infty)\times I^+$ it follows that the mapping 
\[
\{(t,s)\in \R^2\mid t>s\}\times I^+\to \R,\ (t,s,x,y)\mapsto c_{t-s}(x,y),
\]
is locally semiconcave. Thus, the family of functions 
\begin{align*}
&(t_1-\ep,t_1+\ep)\times B_r(\gamma(t_1))\to \R,
\\[10pt]
&(t,x)\mapsto T_su(y)+c_{t-s}(y,x),\ (s,y)\in [t_1-\delta,t_1-2\ep]\times K,
\end{align*}
is uniformly locally semiconcave in the sense of \cite{Fathi/Figalli}, Definition A15. Thus, the infimum (if finite) is also locally semiconcave (\cite{Fathi/Figalli}, Corollary A14). However, the infimum is, by \eqref{c1}, precisely
 $T_tu(x)$. This completes the proof.
\end{proof}

 \section{Lax-Oleinik evolution for $C^1$ Lipschitz functions}
In this section, we aim to study the Lax-Oleinik evolution for $C^1$ Lipschitz functions. As already mentioned in Remark \ref{vv}, it is known that for a Tonelli-Lagrangian on a possibly non-compact manifold $N$, the Lax-Oleinik evolution of a Lipschitz function $f$ is locally semiconcave. This follows from the fact that, locally in $(t,x)$, the infimum in the definition of $T_tf(x)$ can be taken over a compact set. Therefore, since the mapping $(0,\infty)\times N^2\to \R, \ (t,y,x)\mapsto h_t(y,x)$ is locally semiconcave, also the family of mappings
\[
(0,\infty)\times N,\ (t,x)\mapsto f(y)+h_t(y,x),\ y\in K,
\]
is uniformly locally semiconcave, implying that their infimum is locally semiconcave (\cite{Fathi/Figalli}, Proposition A16). In the Lorentzian case, the same issue as explained in Remark \ref{yy} arises, since the functions $c_t(x,y)$ are not even finite or continuous. Again we have to show that, locally in $(t,x)$, the $y$ in the infimum in the definition of $T_tf(x)$ can be taken uniformly bounded away from $\partial J^-(x)$. Under additional assumptions, we will prove that the Lax-Oleinik evolution of a $C^1$ Lipschitz function is locally semiconcave. Let us now state the precise formulation of our main result:

\begin{theorem}\label{cc}
Let $f:M\to \R$ be $C^1$ and Lipschitz. Let $x_0\in M$, and suppose that there exists $r>0$ such that for every $x\in \overbar B_{r}(x_0)$, $d_xf\in T_x^*M\backslash \C_x^*$.
    
    Then, the function $Tf$ is locally semiconcave on $(0,\infty)\times B_r(x_0)$. 
    Furthermore, if $x\in B_r(x_0)$ and $t>0$, there is $y\in M$ such that $T_tf(x)=f(y)+c_t(y,x)$, and we necessarily have $y\in I^-(x)$. If $\gamma:[0,t]\to M$ is a minimizer with $\gamma(0)=y$ and $\gamma(t)=x$ then $(\partial_t c_t(y,x),\frac{\partial L}{\partial v}(x,\dot \gamma(t)))$ is a super-differential for $Tf$ at $(t,x)$.
    Finally, $Tf$ is differentiable at $(t,x)$ if and only if the there is a unique minimizing curve $\gamma:[0,t]\to M$ with $\gamma(t)=x$ such that $T_tf(x)=f(\gamma(0))+c_t(\gamma(0),x)$.
\end{theorem}

Recall (Notation \ref{notation}) that $B_r(x_0)$ denotes the open ball w.r.t.\ the complete metric $h$. As in Section 2, we need some preliminary steps for the proof. We first give the following definition.
\begin{definition}\rm \label{e1}
For $K>0$ let $C(K)>0$ such that
\begin{align*}
L(x,v)\geq K|v|_h-C(K) \text{ for all } (x,v)\in \C.
\end{align*}
Such a constant exists thanks to the uniform superlinearity of $L$ on $\C$ (due to \eqref{splitting}).
\end{definition}

The proof of the next lemma is similar to the proof of Proposition 3.3 in \cite{FathiWK}.

\begin{lemma}\label{t}
   Let $f:M\to \R$ be an $L$-Lipschitz function. Then $Tf$ is everywhere finite. Given $(t,x)\in (0,\infty)\times M$, we have
   \begin{align*}
       T_tf(x)=\inf\{f(y)+c_t(y,x)\mid d_h(y,x) \leq C(L+1)t\},
   \end{align*}
   and the infimum is attained.
   Furthermore,
   \begin{align}
       T_tf(x)<\inf\{f(y)+c_t(y,x)\mid d_h(y,x)> C(L+1)t+1\}. \label{uu}
   \end{align}
\end{lemma}
\begin{proof}
It is clear that $T_tf(x)\leq f(x)<\infty$ for all $(t,x)$. If $(t,x)\in (0,\infty)\times M$, and if $\gamma:[0,t]\to M$ is a causal curve with $\gamma(t)=x$, we have
\begin{align*}
    f(\gamma(0))+\int_0^t L(\gamma(s),\dot \gamma(s))\, ds 
    &\geq f(\gamma(0))+ \int_0^t L |\dot \gamma(s)|_h\, ds-C(L)t
    \\
    &\geq f(\gamma(0))+ L d_h(\gamma(0),\gamma(t))-C(L)t
    \\
    &\geq f(x)-C(L)t.
\end{align*}
Taking the infimum over all possible causal curves, we obtain $T_tf(x)\geq f(x)-C(L)t$, proving the first part of the lemma.

Moreover, if $\ep>0$ and $\gamma:[0,t]\to M$ is a curve with $T_tf(x)\geq f(x)+\int_0^t L(\gamma(s),\dot \gamma(s))\, ds-\ep$, then $\gamma$ must be causal and the same computation as above yields
\begin{align*}
    f(x)\geq T_tf(x)\geq f(x)+d_h(\gamma(0),x)-C(L+1)t-\ep.
\end{align*}
Hence, $d_h(\gamma(0),x)\leq C(L+1)t+\ep$. This proves \eqref{uu} (for $\ep=1$). In particular, any minimizing sequence $(y_k)$ for $T_tf(x)$ is necessarily precompact and converges to some $y\in M$ with $d_h(y,x)\leq C(L+1)t$. By the continuity of $f$, the closedness of $J^+$, and the continuity of $c_t$ on $J^+$, we obtain
\[
T_tf(x)=f(y)+c_t(y,x).
\]
This completes the proof of the lemma.
\end{proof}

\begin{lemma} \label{bbb}
Let $f:M\to \R$ be a Lipschitz function. Then $Tf$ is lower semi-continuous. Moreover, if $(t,x)\in (0,\infty)\times M$ and the infimum in the definition of $T_tf(x)$ is attained at some $y\in I^-(x)$, then $Tf$ is continuous at $(t,x)$.
\end{lemma}
\begin{proof}
Let $(t,x)\in (0,\infty)\times M$, and let $(t_k,x_k)$ a sequence converging to $(t,x)$. For each $k$, choose $y_k\in J^-(x_k)$ such that
\[T_{t_k}f(x_k)=f(y_k)+c_{t_k}(y_k,x_k).\]
By the preceding lemma, the sequence $(y_k)$ is relatively compact. Thus, by passing to a subsequence if necessary, we can assume that $y_k\to y$. Since $J^+$ is closed, $(y,x)\in J^+$. This, together with the continuity of $c$ gives
\begin{align}
\lim_{k\to \infty} T_{t_k}f(x_k)= f(y)+c_t(y,x)\geq T_tf(x). \label{p}
\end{align}
Since this holds true for any sequence, this proves the lower semi-continuity. 

Now, suppose that the infimum in the definition of $T_tf(x)$ is attained at some point $y\in I^-(x)$. For large $k$, it follows that $y\in I^-(x_k)$. Therefore, we have
\begin{align*}
T_{t_k}f(x_k)\leq f(y)+c_{t_k}(y,x_k)\xrightarrow{k\to \infty} f(y)+c_t(y,x)=T_tf(x).
\end{align*}
Together with \eqref{p}, this proves the continuity at $(t,x)$.
\end{proof}

\begin{proposition}
Let $f:M\to \R$ be $C^1$ and $L$-Lipschitz. Let $x_0\in M$ and suppose that $d_{x_0}f\in T_{x_0}^*M\backslash \C_{x_0}^*$, i.e.\ there exists $v\in \C_{x_0}$ with $d_{x_0}f(v)>0$. Given $t_0>0$, there exists a constant $\delta>0$ and a neighbourhood $U$ of $x_0$ such that, for every $x\in U$ and $t\geq t_0$, it holds
\begin{align}
        T_tf(x)<\inf\{&f(y)+c_t(y,x)\mid  \nonumber 
        \\
        &d_h(y,x)> C(L+1)t+1 \text{ or }  d_h(y,x) < \delta\}. \label{e''}
    \end{align}
\end{proposition}

\begin{proof}
 By Lemma \ref{t}, we know that for all $(t,x)\in (0,\infty)\times M$, the following holds:
 \begin{align*}
        T_tf(x)<\inf\{f(y)+c_t(y,x)\mid d_h(y,x)> C(L+1)t+1\}.
    \end{align*}
    Thus, it suffices to prove that there exists a neighbourhood $U$ of $x_0$ and $\delta>0$ such that
    \begin{align*}
        T_tf(x)<\inf\{f(y)+c_t(y,x)\mid d_h(y,x)<\delta\}
    \end{align*}
    for $x\in U$, $t\geq t_0$.

  Let $v\in \op{int}(\C_{x_0})$ such that $d_{x_0}f(v)>0$ (such a $v$ exists by the assumption $d_{x_0}f\in T_{x_0}^*M\backslash \C_{x_0}^*$). Let $(V,\phi)$ be a chart around $x_0$ and let $w:=d_{x_0}\phi(v)$ and $z_0:=\phi(x_0)\in \R^n$. It is easy to see that there exist two constants $r,T>0$ such that $B_{r+T|w|}(z_0)\subseteq \subseteq \phi(V)$,
  \begin{align}
  \inf_{z\in B_{r+T|w|}(z_0)} d_z(f\circ \phi^{-1})(w)=:M>0, \label{a11}
  \end{align}
  and such that, for $z\in B_r(z_0)$, the curve
    \begin{align*}
        [0,T]\to M,\ s\mapsto \phi^{-1}(z-Tw+sw),
    \end{align*}
    is timelike. In particular, $\phi^{-1}(z-sw)\in I^-(\phi^{-1}(z))$ sor $0<s\leq T$.
   
    Let $z\in B_r(z_0)$ and $t\geq t_0$. For $0<s\leq T$, denoting by $C_1$ a Lipschitz constant of $\tau \circ \phi^{-1}$ on $B_{r+T|w|}(z_0)$, we have
    \begin{align*}
        c_t(\phi^{-1}(z-sw),\phi^{-1}(z))
        \leq \frac{(\tau(\phi^{-1}(z))
        -\tau(\phi^{-1}(z-sw)))^2}{t}
        &\leq \frac{C_1^2  s^2|w|^2}{t}
        \\[10pt]
        &\leq \frac{C_1^2  s^2|w|^2}{t_0}.
    \end{align*}
    Using \eqref{a11} and the mean value theorem, we obtain
    \begin{align}
        &f(\phi^{-1}(z-sw))+c_t(\phi^{-1}(z-sw),\phi^{-1}(z)) \nonumber
        \\
        &\leq f(\phi^{-1}(z))-sM+\frac{C_1}{t_0}s^2|w|^2. \label{g}
    \end{align}
  Now, pick $0<s\leq T$ such that $-sM+\frac{C_1}{t_0}s^2|w|^2<0$. Then, if $\delta>0$ is such that 
    \[
    |f(y)-f(x)| <  -sM+\frac{C_1}{t_0}s^2|w|^2 \text{ for all } x,y\in B_{2\delta}(x_0) , 
    \]
 inequality \eqref{g} implies 
    \begin{align*}
        T_tf(\phi^{-1}(z))< \inf\{f(y)+c_t(y,\phi^{-1}(z))\mid d_h(y,\phi^{-1}(z))\leq \delta\}
  \end{align*}
 for $z\in B_r(z_0)$ such that and $\phi^{-1}(z)\in B_\delta(x_0)$. This proves \eqref{e''} for $U:=\phi^{-1}(B_r(z_0))\cap B_\delta(x_0)$.
\end{proof}

\begin{corollary}
    Let $f:M\to \R$ be $C^1$ and $L$-Lipschitz. Let $x_0\in M$ and suppose that $r>0$ is such that for every $x\in \overbar B_{r}(x_0)$, $d_xf\in T_x^*M\backslash \C_x^*$, i.e. there is $v\in \C_x$ with $d_xf(v)>0$.
    
    Given $t_0>0$, there exists a constant $\delta>0$ such that for every $x\in \overbar B_r(x_0)$ and $t\geq t_0$ we have
    \begin{align}
        T_tf(x)<\inf\{&f(y)+c_t(y,x)\mid \nonumber
        \\
        &d_h(y,x)> C(L+1)t+1 \text{ or }  d_h(y,x) < \delta\}. \label{e'}
    \end{align} 
\end{corollary}
\begin{proof}
This follows immediately from the previous proposition and the compactness of $\overbar B_r(x_0)$. 
\end{proof}

\begin{proposition}\label{i1}
   Let $f:M\to \R$ be $C^1$ and $L$-Lipschitz. Let $x_0\in M$ and suppose that $r>0$ is such that for every $x\in \overbar B_{r}(x_0)$, $d_xf\in T_x^*M\backslash \C_x^*$.
    
    Given $t_1>t_0>0$, there exists a constant $\delta>0$ such that for every $x\in B_r(x_0)$ and $t_1\geq t\geq t_0$ we have
    \begin{align*}
        T_tf(x)<\inf\{f(y)+c_t(y,x)\mid d_h(y,x)> C(L+1)t+1 \text{ or }   d(y,x) <\delta\}.
    \end{align*} 
\end{proposition}
\begin{proof}
    Recall that, from the above corollary, there exists $\delta'>0$ such that for $(t,x)\in [t_0,\infty)\times \overbar B_r(x_0)$ we have
    \begin{align}
        T_tf(x)<\inf\{&f(y)+c_t(y,x)\mid \nonumber
        \\
        &d_h(y,x)> C(L+1)t+1 \text{ or }  d_h(y,x) < \delta'\}. \label{dd}
    \end{align}
Let us assume, by contradiction, that there is a sequence $(t_k,x_k)\in [t_0,t_1]\times B_r(x_0)$ and a sequence $(y_k)$ with $\delta'\leq d_h(y_k,x_k)\leq C(L+1)t_k+1$ and $d(y_k,x_k)\to 0$ such that
\begin{align}
f(y_k)+c_{t_k}(y_k,x_k)\leq T_{t_k}f(x_k)+1/k. \label{e}
\end{align}
In particular, $y_k\in J^-(x_k)$.
By the completeness of $M$ w.r.t.\ the metric $h$, it follows that, along a subsequence that we do not relabel, $x_k\to x\in \overbar B_r(x_0)$, $y_k\to y\in \partial J^-(x)\backslash \{x\}$ and $t_k\to t\in [t_0,t_1]$. 

We claim that 
\begin{align}
T_tf(x)=\min\{f(z)+c_t(z,x)\mid z\in \partial J^-(x)\backslash \{x\}\}. \label{q}
\end{align}
In fact, if $Tf$ is continuous at $(t,x)$, then it follows from \eqref{e} that
\begin{align*}
T_tf(x)=f(y)+c_t(y,x).
\end{align*}
Since $y\in \partial J^-(x)\backslash \{x\}$, this proves \eqref{q} in the first case. 
In the second case, assume that $Tf$ is not continuous at $(t,x)$. By Lemma \ref{t} and \eqref{dd}, there exists $z\in J^-(x)\backslash \{x\}$ attaining the infimum in the definition of $T_tf(x)$. However, due to the fact that $Tf$ is discontinuous at $(t,x)$, Lemma \ref{bbb} implies that $z\in \partial J^-(x)\backslash \{x\}$, proving the claim. We have proved that there always exists a minimizer in \eqref{q}, which we will denote by $z\in \partial J^-(x)$.

Let $\gamma:[0,t]\to M$ be a causal null geodesic connecting $z$ with $x$. Choose $\xi\in T_zM$ with $g(\dot \gamma(0),\xi)>0$. Let $\xi:[0,t]\to TM$ the parallel vector field along $\gamma$ with $\xi(0)=\xi.$ Consider the vector field $\tilde \xi(s):=(t-s)\xi(s)$. We can now pick a smooth variation $c:(-\ep,\ep)\times [0,t]\to M$ of $\gamma$ such that its variational vector field is $\tilde \xi$ and such that $c(r,t)=\gamma(t)=x$ for all $r\in (-\ep,\ep)$. Our goal is to show that, for $r$ small enough,
\begin{align*}
    f(c(r,0))+c_t(c(r,0),x)<f(z)+c_t(z,x),
\end{align*}
contradicting the definition of $z$ and, thus, proving the claim. 

We compute
    \begin{align*}
        \frac{d}{dr}\Big\vert_{r=0} g(\partial_s c(r,s),\partial_s c(r,s))
        &= 2 g\bigg(\frac{\nabla}{dr}\Big\vert_{r=0}\frac{\partial c}{\partial s}(r,s),\frac{\partial c}{\partial s}(0,s)\bigg)
        \\[10pt]
        &=
        2 g\bigg(\frac{\nabla}{ds}\frac{\partial c}{\partial r}(0,s),\dot \gamma(s)\bigg)
        \\[10pt]
        &=
        2 g\bigg(\frac{\nabla}{ds}((t-s)\xi(s)),\dot \gamma(s)\bigg)
        \\[10pt]
        &= -2 g(\xi(s),\dot \gamma(s))
        \\[10pt]
        &= -2 g(\xi(0),\dot \gamma(0))=:-2a<0,
    \end{align*}
    where in the step to the last line we used the fact that $\xi$ and $\dot \gamma$ are parallel along $\gamma$.
    Thus, Taylor expansion and the fact that $\gamma$ is a null geodesic yield
    \begin{align}
        g(\partial_s c(r,s),\partial_s c(r,s))\leq -|\dot \gamma(s)|^2_g -2ar +O(r^2)\leq -ar \label{hhhh}
    \end{align}
    for small values $r>0$. In particular, $c(r,\cdot)$ is either (future-directed) timelike or past-directed timelike for small $r$.
    Since $z\neq x$ and $z\in J^-(x)$ we must have $z\notin J^+(x)$. Thus $z\approx c(r,0)\notin J^+(x)$ for small $r$, so that $c(r,\cdot)$ is in fact (future-directed) timelike for small $r$.
This, together with
\begin{align*}
    L_g(c(r,\cdot))\geq t \sqrt{ar},
\end{align*}
as follows from \eqref{hhhh}, gives
\begin{align*}
    d(c(r,0),x)\geq  t\sqrt{ar}.
\end{align*}
On the other hand, since $\partial_r c(0,0)=t\xi$, it holds $d_h(c(r,0),z)\leq 2t|\xi|_hr$, if $r$ is small.

Recall that $L$ is a Lipschitz constant for $f$. It follows for small $r>0$ that
    \begin{align*}
        &f(c(r,0))+c_t(c(r,0),x)
        \\[10pt]
        \leq \ & f(z) +L d_h(c(r,0),z)+\frac{(\tau(x)-\tau(c(r,0))-d(c(r,0),x))^2}{t}
        \\[10pt]
        \leq \ &f(z) +2Lt|\xi|_hr +\frac{\Big([\tau(x)-\tau(z)]+[\tau(z)-\tau(c(r,0))]-[d(c(r,0),x)]\Big)^2}{t}
        \\[10pt]
        \leq \ &f(z)+2Lt|\xi|_hr +\frac{\Big(\tau(x)-\tau(z)+2\op{Lip}(\tau)t|\xi|_hr-t\sqrt{ar}\Big)^2}{t}
    \end{align*}
   In the last line, we abused the notation $\op{Lip}(\tau)$ to denote a Lipschitz constant of $\tau$ in a small neighbourhood of $z$. For small values of $r$, the sum of the last two quantities in the bracket is less than $-r^{3/4}$. Thus,
    \begin{align*}
        f(c(r,0))+c_t(c(r,0),x)
         \leq 
        f(z)+2Lt|\xi|_hr +\frac{(\tau(x)-\tau(z)-r^{3/4})^2}{t}
    \end{align*}
    Since $z\in \partial J^-(x)$, it holds $c_t(z,x)=(\tau(x)-\tau(z))^2/t$. Then, factoring out gives
    \begin{align*}
         &f(c(r,0))+c_t(c(r,0),x)
         \\[10pt]
         \leq \ & f(z)+2Lt|\xi|_hr+ c_t(z,x)-\frac 2t(\tau(x)-\tau(z))r^{3/4}+\frac{r^{3/2}}t \nonumber
        \\[10pt]
        \leq\ 
        & f(z)+c_t(z,x) +\left(2Lt|\xi|_hr-\frac 2t(\tau(x)-\tau(z))r^{3/4}+\frac{r^{3/2}}t\right).
    \end{align*}
    Clearly, this quantity is less than $f(z)+c_t(z,x)$, if $r$ is small enough (note that at this point we use $z\in J^-(x)\backslash \{x\}$ since then $\tau(x)-\tau(z)>0$). This is a contradiction to the definition of $z$.
\end{proof}

We are now in the position to prove Theorem \ref{cc}. With the tools we have in hand by now, the proof is similar to the proof of Corollary 4.23 in \cite{FathiHJ}.

\begin{proof}[Proof of Theorem \ref{cc}]
Let $t_1>t_0>0$ be given. The foregoing lemma states that there exists $\delta>0$ such that for all $x\in B_r(x_0)$ and $t_1\geq t\geq t_0$, we have
    \begin{align*}
        T_tf(x)=\inf\{f(y)+c_t(y,x)\mid d_h(y,x)\leq C(L+1)t+1,\  d(y,x) \geq\delta\}.
    \end{align*} 
One argues as in the proof of Theorem \ref{m} to show that $Tf$ is locally semiconcave on $(t_0,t_1)\times B_r(x_0)$, and, hence, locally semiconcave on $(0,\infty)\times B_r(x_0)$.

    If $x\in B_r(x_0)$ and $t>0$ are given, and $y\in M$ is such that $T_tf(x)=f(y)+c_t(y,x)$, then the above proposition tells us that necessarily $y\in I^-(x)$. Let $\gamma:[0,t]\to M$ be a (timelike) minimizer connecting $y$ to $x$. It is a consequence from Lemma \ref{ddd} that a super-differential for $Tf$ at $(t,x)$ is given by
    \begin{align*}
        \left(\partial_t c_{t}(x,y),\frac{\partial L}{\partial v}(y,\dot \gamma(t))\right).
    \end{align*}

    If $Tf$ is differentiable at $(t,x)$, then there exists a unique super-differential. Given two minimizers $\gamma_1,\gamma_2:[0,t]\to M$ with $\gamma_1(t)=\gamma_2(t)=x$ such that both $\gamma_1(0)$ and $\gamma_2(0)$ attain the infimum in the definition of $T_tf(x)$, they must be timelike. Hence, from what we just proved,
    
\begin{align*}
     &\left(\partial_t c_{t}(\gamma_1(0),x),\frac{\partial L}{\partial v}(x,\dot \gamma_1(t))\right)\in \partial^+(Tf)(t,x),\
     \\[10pt]
     &\left(\partial_t c_{t}(\gamma_2(0),x),\frac{\partial L}{\partial v}(x,\dot \gamma_2(t))\right)\in \partial^+(Tf)(t,x).
\end{align*}
In particular, both super-differentials have to coincide. Since the Legendre transform is a diffeomorphism by Lemma \ref{hhhhh}, it follows that $\dot \gamma_1(t)=\dot \gamma_2(t)$. The unique integrability of the Euler-Lagrange flow implies $\gamma_1=\gamma_2$.

    Conversely, assume that there exists exactly one minimizing curve $\gamma:[0,t]\to M$ with $\gamma(t)=x$ such that $T_tf(x)=f(\gamma(0))+c_t(\gamma(0),x)$. Let $y:=\gamma(0)$. Since $Tf$ is locally semiconcave in a neighborhood of $(t,x)$, it is differentiable almost everywhere and the set of super-differentials at $(t,x)$ is given by the closure of the convex hull of reaching gradients:
    \begin{align*}
        \partial^+(Tf)(t,x)=\overline {\op{conv}\left\{\lim_{k\to \infty} \left(\partial_t c_{t_k}(y_k,x_k),\frac{\partial L}{\partial v}(x_k,\dot \gamma_k(t_k))\right)\right\}},
    \end{align*}
    where we take all sequences $(t_k,x_k)\in  (0,\infty)\times B_r(x_0)$ converging to $(t,x)$ such that $Tf$ is differentiable at $(t_k,x_k)$. Here, $y_k:=\gamma_k(0)$ and $\gamma_k:[0,t_k]\to M$ is the unique minimizing curve with $\gamma_k(t_k)=x_k$ such that $T_{t_k}f(x_k)=f(y_k)+c_{t_k}(y_k,x_k)$.
    
    Now, let $(t_k,x_k)$ be such a sequence. Note that, by the foregoing proposition, we have $d_h(y_k,x_k)\leq C(L+1)t+2$ and $d(y_k,x_k)\geq \delta$ for sufficiently large $k$ and some $\delta>0$. Hence, along a subsequence that we do not relabel, $y_k\to \tilde y\in I^-(x)$. Due to the continuity of $Tf$ at $(t,x)$ it follows that $T_tf(x)=f(\tilde y)+c_t(\tilde y,x)$. We are now in position to apply Corollary \ref{ee}. Let $\tilde \gamma:[0,t]\to M$ be the minimizing (timelike) curve guaranteed by the corollary, and note that $\tilde \gamma(0)=\tilde y$, $\tilde \gamma(t)=x$ and $T_tf(x)=f(\tilde \gamma(0))+c_t(\tilde \gamma(0),x))$. In particular, by assumption, we must have $\tilde \gamma=\gamma$ and $\tilde y=y$. Then, clearly, $\partial_t c_{t_k}(y_k,x_k)\to \partial_t c_t(y,x)$ as $k\to \infty$. Moreover, by the corollary,
    \begin{align*}
        \frac{\partial L}{\partial v}(x_k,\dot \gamma_k(t_k)) \to \frac{\partial L}{\partial v}(x,\dot \gamma(t)) \text{ as } k\to \infty.
    \end{align*}
    Thus, the super-differential $\partial^+(Tf)(t,x)$ reduces to the singleton \linebreak $\{(\partial_t c_{t}(y,x),\frac{\partial L}{\partial v}(x,\dot \gamma(t)))\}$, that is, $Tf$ is differentiable at $(t,x)$.
\end{proof}

\section{Application to Optimal transport}
In this section we want to prove our main theorem. We start with the basic definition and well-known facts, also to fix our notation.

Let $\mu_0,\mu_1\in {\cal P}(M)$ be two Borel probability measures. A \emph{coupling} of $\mu_0$ and $\mu_1$ is a Borel probability measure $\pi\in {\cal P}(M\times M)$ such that its marginals are $\mu_0$ and $\mu_1$, that is,
\[ 
\forall B\subseteq M \text{ Borel }: \mu_0(B)=\pi(B\times M) \text{ and } \mu_1(B)=\pi(M\times B).
\]
In other words, $\mu_0$ and $\mu_1$ are the push-forward measures of $\pi$ w.r.t.\ the projections $(x,y)\mapsto x$ and $(x,y)\mapsto y$.
The set of couplings is denoted by $\Gamma(\mu_0,\mu_1)$ and the \emph{total cost} for the cost function $c$ is defined as
\begin{align*}
   C(\mu_0,\mu_1)=\inf\bigg\{\int_{M\times M} c(x,y)\, d\pi(x,y)\mid \pi \in \Gamma(\mu_0,\mu_1)\bigg\}.
\end{align*}  
A coupling is said \emph{optimal} if it minimizes the right-hand side of the above expression. It is a standard result in the theory of optimal transportation, using some compactness arguments (Prokhorov's theorem), that there always exists an optimal coupling.

\textbf{From now on}, fix $\mu_0$ and $\mu_1$ as above. We assume that $C(\mu_0,\mu_1)<\infty$.

\begin{definition}\rm
A \emph{dynamical optimal coupling} is a Borel probability measure $\Pi$ on the metric space of minimizing curves $\Gamma^1$ such that $(e_0,e_1)_\#\Pi\in \Gamma(\mu_0,\mu_1)$ is an optimal coupling.  

Here, for $t\in [0,1]$, the map $e_t$ is defined by $e_t(\gamma):=\gamma(t)$.
\end{definition}

\begin{lemma}\label{a0}
    If $\pi\in \Gamma(\mu_0,\mu_1)$ is an optimal coupling, then there exists a dynamical optimal coupling $\Pi$ with $(e_0,e_1)_\#\Pi=\pi$.
\end{lemma}
\begin{proof}
    We borrow the proof from \cite{Villani}, 7.16. Consider the map
    \[
    (e_0,e_1):\Gamma^1 \to J^+,\ \gamma \mapsto (\gamma(0),\gamma(1)).
    \]
    Then $(e_0,e_1)^{-1}(x,y)$ is compact for all $(x,y)\in J^+$ thanks to Corollary \ref{ee}. Moreover, both $\Gamma^1$ and $J^+$ are complete separable metric spaces thanks to Corollary \ref{eeef} and the closedness of $J^+$. Thus, a well-known selection theorem (see \cite{selection}) states that there exists a measurable map
    $S:J^+\to \Gamma^1$ such that $S(x,y):[0,1]\to M$ is a minimizing curve connecting $x$ to $y$.
    Since the optimal coupling $\pi$ must be supported on $J^+$ (here we use the fact that $C(\mu_0,\mu_1)$ is finite), $S$ is defined $\pi$-a.e., hence we can define the measure
    \[
    \Pi:=S_\#\pi.
    \]
    It is easy to check that this measure is indeed a  dynamical optimal coupling with $(e_0,e_1)_\#\Pi=\pi$.
\end{proof}

Now, \textbf{for the remainder of this section}, let $\pi\in \Gamma(\mu_0,\mu_1)$ be an optimal coupling, and let $\Pi$ be a dynamical optimal coupling associated with $\pi$, i.e. $(e_0,e_1)_\#\Pi=\pi$. We define the curve of probability measures
\begin{align}
    [0,1]\to {\cal P}(M),\ t\mapsto \mu_t:=(e_t)_\#\Pi. \label{aaaa}
\end{align}
This curve is also referred to as \emph{displacement interpolation}.
For $0\leq s <t\leq 1$, the probability measure $\pi_{s,t}:=(e_s,e_t)_\#\Pi$ clearly couples $\mu_s$ with $\mu_t$. In fact, one can easily prove as in \cite{Villani}, Theorem 7.21 (i)$\Rightarrow (ii)$, that it is an optimal coupling (for $\mu_s$ and $\mu_t$) in the optimal transportation problem associated to the cost function $c_{t-s}$.

The following two definitions are adapted (with slight modifications) from \cite{Fathi/Figalli}, Definition 2.1 and 2.2.

\begin{definition}\rm
A pair of functions $\varphi,\psi:M\to \overbar \R$ is called \emph{$c$-subsolution} if
\begin{align*}
    \psi(y)-\varphi(x)\leq c(x,y) \text{ for all } x,y\in M.
\end{align*}
Here, we use the convention $\pm\infty\mp \infty:=-\infty$.
Equivalently(!), $\psi\leq T_1\varphi$.
\end{definition}

\begin{definition}\rm
    We say that a $c$-subsolution is \emph{$(c,\pi)$-calibrated} if 
    \begin{align*}
        \psi(y)-\varphi(x)=c(x,y)\ \text{ for } \pi\text{-a.e. } (x,y).
    \end{align*}
\end{definition}

\begin{remark}\rm
\begin{enumerate}[(i)]
\item
Note that this definition differs from the one in \cite{Fathi/Figalli} since we do not assume $\psi$ and $\varphi$ to be Borel and $\varphi$ and $\psi$ may attain the value $-\infty$ and $\infty$, respectively.
\item By definition, in the Riemannian case, $\varphi$ and $\psi$ map to $\R \cup \{\infty\}$ and $\R\cup \{-\infty\}$, respectively. The equivalence between the definition of a $c$-subsolution and the requirement $\psi\leq T_1\varphi$ is then trivial. However, with our conventions, it is not difficult to check the equivalence.

However, there is no equivalence in being calibrated and satisfying $\psi=T_1\varphi$. For instance, if $\varphi,\psi \equiv \infty$, we have $\psi= T_1\varphi$, but $(\varphi,\psi)$ is not $(c,\pi)$-calibrated. 
\end{enumerate}
\end{remark}

Next, \textbf{for the remainder of this chapter}, suppose that $(\varphi,\psi)$ is a $(c,\pi)$-calibrated pair and that $\pi(I^+)=1$. Then the set 
\begin{align}
    \{(x,y)\in I^+\mid \psi(y)-\varphi(x)=c(x,y)\} \label{hhhhhh}
\end{align}
is of full $\pi$-measure (in particular, $\psi(y),\varphi(x)\in \R$ for all these $(x,y)$!).

In \cite{Metsch} it was shown that, under suitable conditions on the measures $\mu_0$ and $\mu_1$, there always exists a $(c,\pi)$-calibrated pair, and any optimal coupling must be concentrated on $I^+$. 

Hence, there exists a Borel set $A\subseteq \Gamma^1$ consisting of timelike minimizers on which $\Pi$ is concentrated and such that
\begin{align*}
    \psi(\gamma(1))-\varphi(\gamma(0))=c(\gamma(0),\gamma(1)) \text{ for all } \gamma \in A.
\end{align*}
Indeed, denoting by $B\subseteq I^+$ a Borel subset of \eqref{hhhhhh} of full $\pi$-measure, we can set $A:=(e_0,e_1)^{-1}(B)$. Note once again that $\psi(\gamma(1)),\varphi(\gamma(0))\in \R$ for all $\gamma \in A$.
By definition, the measure $\mu_t$ is then concentrated on the set
 \begin{align*}
    A_t:= \{\gamma(t)\mid \gamma \in A\}\subseteq M.
 \end{align*}

Using the definitions, the following lemma is well-known and easy to prove, see \cite{Villani}, Lemma 7.36. 

\begin{lemma}\label{jjjjj}
   For any $0\leq s\leq t\leq 1$ the pair
    \begin{align*}
    &\varphi_s(x):=\inf_{z\in M} \varphi(z)+c_s(z,x) \text{ and }
    \\[10pt]
    &\psi_t(y):=\sup_{z\in M} \psi(z)-c_{1-t}(y,z).
\end{align*}
is a $c_{t-s}$-subsolution. Moreover, if $\gamma \in A$, then
\begin{align*}
    \psi_t(\gamma(t))-\varphi_s(\gamma(s))=c_{t-s}(\gamma(s),\gamma(t)).
\end{align*}
\end{lemma}
\begin{proof}
   For a proof, see \cite{Villani}. The lemma also follows from the following more general statement (whose proof is similar).
\end{proof}

\begin{lemma}\label{h1}
Let $\gamma \in A$ and $s,t\in [0,1]$.
\begin{enumerate}[(i)]
    \item 
    If $0\leq \tau\leq 1-s$ and $0\leq r\leq s+\tau$ then
    \[
    \hat T_{r}\varphi_{s+\tau}(\gamma(s+\tau-r))=\varphi(\gamma(0))+c_{s+\tau-r}(\gamma(0),\gamma(s+\tau-r)).
    \]
    \item 
   If $0\leq \tau\leq t$ and $0\leq r\leq 1-t+\tau$ then
    \[
    \ T_{r}\psi_{t-\tau}(\gamma(t-\tau+r))=\varphi(\gamma(0))+c_{t-\tau+r}(\gamma(0),\gamma(t-\tau+r)).
    \]
    \item 
    If $\tau_1,\tau_2\geq 0$ are such that $s+\tau_1\leq t-\tau_2$ then
     the pair $(\hat T_{\tau_2}\varphi_{s+\tau_2},T_{\tau_1}\psi_{t-\tau_1})$ is still $(c_{t-s},\pi_{s,t})$-calibrated.
     Additionally, if $\gamma \in A$, 
     \[
     T_{\tau_1}\psi_{t-\tau_1}(\gamma(t))-
     \hat T_{\tau_2}\varphi_{s+\tau_2}(\gamma(s))=c_{t-s}(\gamma(s),\gamma(t)).
     \]
\end{enumerate}
\end{lemma}
\begin{proof}
The proof boils down to elementary properties of the Lax-Oleinik semigroups and the fact that $\psi(\gamma(1))-\varphi(\gamma(0))=c(\gamma(0),\gamma(1))$. The details can be found in subsection \ref{proofs}.
\end{proof}
 
Using Theorem \ref{m}(i) and the fact that $T_1\varphi(\gamma(1))=\varphi(\gamma(0))+c(\gamma(0),\gamma(1))$, it is not difficult to show that, for $s\in (0,1)$, $\varphi_s$ is locally semiconcave on a neighborhood of $A_s$. Similarly, by Theorem \ref{m}(ii), for $t\in (0,1)$, the function $\psi_t$ is locally semiconvex on a neighborhood of $A_t$. While we do not require this result, the following lemma establishes a similar statement.

\begin{corollary}\label{s}
Fix $0<s<t<1$. For $0<\tau\leq 1-s$ the function $\hat T_\tau\varphi_{s+\tau}$ is locally semiconvex on a neighbourhood of $A_{s}$. 
For $0<\tau\leq t$ the function $T_\tau\psi_{t-\tau}$ is locally semiconcave on a neighbourhood of $A_t$.
\end{corollary}
\begin{proof}
    If $\gamma\in A$, then from part (i) of the preceding lemma (applied to $r=s+\tau$ and to $r=0$) and the fact that $\gamma$ is a minimizer,
    \begin{align*}
        \hat T_{s+\tau}\varphi_{s+\tau}(\gamma(0)) =\varphi_{s+\tau}(\gamma(s+\tau)) -c_{s+\tau}(\gamma(0),\gamma(s+\tau))
    \end{align*}
   Hence, we can apply Theorem \ref{m}(ii) to $u:=\varphi_{s+\tau}$, $x_0:=\gamma(s+\tau)$, $y_0:=\gamma(0)$ and $t_0:=s+\tau$, proving that $\hat T_\tau\varphi_{s+\tau}$ is locally semiconvex in a neighbourhood of $\gamma(s)$. Since $\gamma$ was arbitrary, this proves the first claim. The second statement can be derived similarly. Indeed, part (ii) of the above lemma applied to $r=1-t+\tau$ and $r=0$
   yields
    \[
    T_{1-t+\tau}\psi_{t-\tau}(\gamma(1))=\psi_{t-\tau}(\gamma-\tau))+c_{1-t+\tau}(\gamma(t-\tau),\gamma(1)).
    \]
   The claim follows from Theorem \ref{m}(i) with $x_0:=\gamma(t-\tau)$ and $y_0:=\gamma(1)$ and $t_0:=1-t+\tau$.
\end{proof}

Let us recall our main result (with a slightly different formulation and with an additional property) in this paper (Theorem \ref{main2}):

\begin{theorem}\label{n}
    Fix $0<s<t<1$. There exists a pair of functions $\Phi_s,\Psi_t:M\to \overbar \R$ that is $(c_{t-s},\pi_{s,t})$-calibrated and such that $\Phi_s$ is $C_{loc}^{1,1}$ on a neighbourhood of $A_s$ while $\Psi_t$ is $C_{loc}^{1,1}$ on a neighbourhood of $A_t$.

    In addition, $\Phi_s=\varphi_s$ on $A_s$ and $\Psi_t=\psi_t$ on $A_t$.
\end{theorem}

    For the proof we need some preparation. Let us first prove a corollary concerning the dual optimal transporation problem (Theorem \ref{main}). Let us denote the total cost of the optimal transportation problem associated with the cost function $c_{t-s}$ by $C^{t-s}$.

\begin{corollary}
    Fix $0<s<t<1$ and assume that $\varphi\in L^1(\mu_0)$ and $\psi\in L^1(\mu_1)$. Then $\Phi_s\in L^1(\mu_s)$, $\Psi_t\in L^1(\mu_t)$ and they are optimal for the dual problem of optimal transportation associated with the cost function $c_{t-s}$, i.e.
    \[
    C^{t-s}(\mu_s,\mu_t)=\int_M \Psi_t(y)\, d\mu_t(y)-\int_M \Phi_s(x)\, d\mu_s(x).
    \]
\end{corollary}
\begin{proof}
Since the pair $(\Phi_s,\Psi_t)$ is $(c_{t-s},\pi_{s,t})$-calibrated, it is enough to show that $\Phi_s\in L^1(\mu_s)$ and $\Psi_t\in L^1(\mu_t)$. 
Since $\mu_s$ is concentrated on $A_s$ and $\Phi_s$ is $C_{loc}^{1,1}$ on some open neighbourhood of $A_s$, it follows that $\Phi_s$ is $\mu_s$-measurable. Similarly, $\Psi_t$ is $\mu_t$-measurable.

    Recall that we denote by $\Pi$ our dynamical optimal coupling and that $\mu_t=(e_t)_\#\Pi$ and $\mu_s=(e_s)_\#\Pi$. 
    For $\gamma \in A$ we have by the theorem and Lemma \ref{h1}
    \[
    \Psi_t(\gamma(t))=\psi_t(\gamma(t))=\varphi(\gamma(0))+c_t(\gamma(0),\gamma(t))=\varphi(\gamma(0))+t c(\gamma(0),\gamma(1)).
    \]
   Thus, the change of variables formula implies
   \begin{align*}
   \int_M |\Psi_t|\, d\mu_t &= \int_{\Gamma^1} |\Psi_t(\gamma(t))|\, d\Pi(\gamma)
   \\[10pt]
   &\leq
   \int_{\Gamma^1} |\varphi(\gamma(0))|\, d\Pi(\gamma)+ t \int_{\Gamma^1} c(\gamma(0),\gamma(1))\, d\Pi(\gamma)
   \\[10pt]
   & =  \int_M |\varphi|\, d\mu_0 +t C(\mu_0,\mu_1)<\infty.
   \end{align*}
   This shows that $\Psi_t\in L^1(\mu_t)$ and, similarly, one proves that $\Phi_s\in L^1(\mu_s)$.

   Thus, we can integrate the identity $\Psi_t(y)-\Phi_s(x)=c_{t-s}(x,y)$, which holds $\pi_{s,t}$-a.e., and obtain that $(\Phi_s,\Psi_t)$ is an optimal pair for the dual formulation.
\end{proof}

\subsection*{Proof of Theorem \ref{n}}

 As mentioned above, to prove Theorem \ref{n} we need some preliminary results. The following proposition extends the results on the regularity of solutions to the Hamilton-Jacobi equation presented in \cite{Bernard}, and it plays the key role in the proof of our main theorem. We split the proof (as given in \cite{Bernard}) into two parts: the proof of the proposition and the proof of Lemma \ref{llllll}.

\begin{proposition}\label{o}
    Let $t_0\in (0,1)$ and $\gamma \in A$. Then, for $\tau>0$ small enough, the function $T_\tau\psi_{t_0-\tau}$ is $C_{loc}^{1,1}$ on an open neighborhood of $\gamma(t_0)$.

   Similarly, for $s_0\in (0,1)$, $\hat T_\tau \varphi_{s_0+\tau}$ is $C^{1,1}_{loc}$ on an open neighbourhood of $\gamma(s_0)$.
\end{proposition}

\begin{lemma}\label{llllll}
   Let $f_i:M\to \R$, $i\in I$, be a family of smooth $L$-Lipschitz functions that is locally bounded in $C^2$. Let $x_0\in M$ and $r>0$ such that
    \begin{align*}
        \{(x,d_xf_i)\mid x\in  B_{r}(x_0),\ i\in I\} \subseteq \subseteq  T^*M\backslash \C^*.
    \end{align*}
    Then there exist constants $0<\ep,\delta<r$ such that:
    
    Given any $0\leq t\leq \ep$ and $i\in I$, the map $\psi^H_{i,t}:U_{i,t}\to B_\delta(x_0),\ x\mapsto \pi\circ \psi^H_t(x,d_xf_i),$ is a diffeomorphism from an open neighbourhood $U_{i,t}\supseteq B_\ep(x_0)$ of $x_0$ and satisfies, for $x\in U_{i,t}$, 
    \begin{align}
        T_tf_i(\psi^H_{i,t}(x))=f(x)+\int_0^t L(\psi^H_{i,s}(x,d_xf_i),\frac{d}{ds}\psi^H_{i,s}(x,d_xf_i))\, ds. \label{c}
    \end{align}
    It follows that $T_tf_i$ is smooth on $B_\delta(x_0)$. In addition, the family of maps $\{T_tf_i\mid 0\leq t\leq \ep,\ i\in I\}$ is locally bounded in $C^2(B_\delta(x_0))$.
\end{lemma}

\begin{remark}\rm
\begin{enumerate}[(a)]
\item
In the lemma, $\pi:T^*M\to M$ is the canonical projection. Our proof uses the inverse function theorem. Since the theorem will be applied for a family of maps, we need some control over the size of the diffeomorphic neighbourhoods.  This is provided by the following lemma, which can be found in \cite{Lang}, Chapter 14, Lemma 1.3.
\item 
For a detailed proof of some variant in the non-compact Riemannian case, which does not use the inverse function theorem, see Appendix B in \cite{Fathi/Figalli/Rifford}.
\item
For the definition and further properties of $C^2$-boundedness, see Subsection \ref{semiconcavity} in the appendix.
\end{enumerate}
\end{remark}

\begin{lemma}
    Let $V\subseteq \R^n$ be open, and let $f:V\to \R^n$ be $C^1$. Let $x_0\in V$ and $R>0$ such that $B_R(x_0)\subseteq V$. Suppose that $C\geq 1$ is such that $f$ is $C$-Lipschitz and $C\geq \norm{Df(x_0)^{-1}}$. Additionally, assume that
    \[
    \norm{Df(x)-Df(x_0)}\leq \frac 1{4C}
    \]
    for all $x\in  B_R(x_0)$. If $r\leq \frac{R}{2C}$ and $B_r(f(x_0))\subseteq W\subseteq B_{\frac{R}{2C}}(x_0)$ is any open set, then $B_{\frac rC}(x_0)\subseteq U:=f^{-1}(W)$, and the map $f_{|U}:U\to W$ is a diffeomorphism. 
\end{lemma}

\begin{proof}[Proof of Lemma \ref{llllll}]

By assumption, choosing $r$ even smaller if neccessary, there is a chart $(\phi,B_{2r}(x_0))$ such that the family $(f_i\circ \phi^{-1})_i$ is $K$-bounded in $C^2$.
 Let $T>0$ be such that 
 \begin{align}
\psi^H\left( [-T,T]\times \overbar C\right)\subseteq T^*B_{2r}(x_0),\ C:={\{(x,d_xf_i)\mid x\in B_r(x_0), i\in I\}}. \label{d1}
 \end{align}
Such a $T$ exists because, by assumption, $\overbar C\subseteq T^*M\backslash \C^*$ and $\psi^H$ is the Hamiltonian flow on $T^*M\backslash \C^*$. We split the proof into three steps. 
\medskip

\noindent \textbf{First step: There exists a constant $0<\delta_0<r$ such that:
If $\delta\leq \delta_0$, then there exists $\ep=\ep(\delta)<\delta$ such that, for any $0\leq t\leq \ep$ and $i\in I$, the map 
    \[\psi^H_{i,t}:U_{i,t}\to B_{\delta}(x_0),\ x\mapsto \pi\circ \psi^H_t(x,d_xf_i),\]
     is a diffeomorphism from an open neighbourhood $B_{r}(x_0)\supseteq U_{i,t}=U_{i,t}(\delta)\supseteq B_{\ep}(x_0)$. Moreover, if $\delta\leq \delta'$, then $\ep(\delta)\leq \ep(\delta')$ and $U_{i,t}(\delta)\subseteq U_{i,t}(\delta')$ for all $0\leq t \leq \ep(\delta)$.}

\noindent Due to \eqref{d1} and since the result is local in nature, we may assume\footnote{We just set $\tilde M:=\phi(B_{2r}(x_0))\subseteq \R^{2n}$ and take $\tilde r>0$ with $B_{2\tilde r}(\phi(x_0))\subseteq \tilde M$. Then we find $\tilde T$ such that \eqref{d1} holds (with $T$ and $r$ replaced by $\tilde T$ and $\tilde r$). If we can prove the result for $\tilde r$ and $\tilde M$, we can go back via the chart to conclude the proof of the step.}
that $M$ is an open subset of $\R^n$,
that the family of maps $f_i:\R^n\supseteq B_{2r}(x_0)\to \R$ is $K$-bounded in $C^2$ and $\psi^H$ (which is now the Hamiltonian flow w.r.t.\ the standard symplectic structure on $\R^{2n}$ w.r.t.\ the Hamiltonian $H$ in local coordinates) is defined on $[-T,T]\times \overbar C$.

 We set 
\[
O:=\psi^H([-T,T]\times \overbar C)\subseteq \R^{2n}.
\]
Let us also define, for $i\in I$, the map
\[
\psi^H_i:(-T,T)\times B_r(x_0)\to \R^n,\ (t,x)\mapsto \pi \circ \psi^H_t(x,Df_i(x)).
\]
	
\noindent \textbf{Claim:} On $(-T,T)\times  B_r(x_0)$, we have
 \[
 D_x\psi^H_{i}(t,x)=\Id+R_i(t,x),
  \]
    where $\sup\{|R_i(t,x)|/|t|\mid (t,x)\in (-T,T)\times  B_r(x_0),\ t\neq 0,\ i\in I\}<\infty.$ Here, $D_x$ stands for differentiation w.r.t\ $x$.
	\medskip
    
	\noindent \textbf{Proof:}
	Let us differentiate on $(-T,T)\times B_r(x_0)$ the identity (which holds by definition of the Hamiltonian flow)
	\[ \psi^H_{i}(t,x)=x+\int_0^t \nabla_p H( \psi^H_{s}(x,Df_i(x)))\, ds,\]
   w.r.t. $x$, yielding
	\begin{align*}
	D_x\psi^H_{i}(t,x)=\Id
    +\int_0^t &D(\nabla_p H)( \psi^H_{s}(x,Df_i(x)))
    \\[10pt]
    &\circ \bigg(D_x \psi^H_{s}(x,Df_i(x))
	+ D_p \psi^H_s(x,Df_i(x))\circ D^2f_i(x)\bigg)\, ds.
        \end{align*}
	We define $R_i(t,x)$ as the integral. The integrand can be bounded by
    \[
    \sup_O \norm{D(\nabla_p H)} \Big(\sup_{ [-T,T]\times \overbar C} \norm{D_x\psi^H_s}+K\sup_{ [-T,T]\times \overbar C}  \norm{D_p \psi^H_s}\Big).
    \]
   Since the bound does not depend on $i$, this proves the claim.
    \hfill \checkmark
    
    Now consider the family of maps
    \[
     \Psi^H_i:(-T,T)\times B_{r}(x_0)\to \R\times \R^n,\ (t,x)\mapsto (t, \psi^H_{i}(t,x)),\ i\in I.
    \]
    The derivative is represented by the matrix
    \[
   D\Psi^H_i(t, x)= \left(
    \begin{array}{c|c}
      1 & 0 \\[3pt]
      \hline 
    \nabla_p H(\psi^H_{t}(x,Df_i(x))) & D_x\psi^H_{i}(t,x) 
    \end{array}
    \right)\in \R^{(1+n)\times (1+n)}.
    \]
    Thanks to the claim and the two facts that (1) the family of maps $(x\mapsto (x,Df_i(x)))_{i\in I}$ is equi-continuous at $x_0$ (because of the $K$-boundedness), and (2) the map $\nabla_pH\circ \psi^H$ is uniformly continuous on the compact subset $[-T,T]\times \overbar C$, we conclude that the family $(D\Psi^H_i)_{i\in I}$ is equi-continuous at $(0, x_0)$. 	

    Moreover, $D \Psi^H_i(0, x_0)^{-1}$ is given by the matrix
\[
   D\Psi^H_i(0, x_0)^{-1}= \left(
    \begin{array}{c|c}
      1 & 0 \\[3pt]
      \hline 
    -\nabla_p H(\psi^H_{t}(x_0,Df_i(x_0))) & \op{I}_n 
    \end{array}
    \right)\in \R^{(1+n)\times (1+n)},
    \]
which is bounded (uniformly in $i$) by some constant $C_1\geq 1$, thanks to the compactness of $\overbar C$. Thanks to the claim and the fact that, easily verified, $\partial_t \psi_i^H(t,x)$ is uniformly bounded on $(-T,T)\times B_r(x_0)$, we can choose $C_1$ such that each $\Psi^H_i$ is $C_1$-Lipschitz.

Let $0<\delta_0<r$ with $B_{2\sqrt{2}C_1\delta_0}(0,x_0)\subseteq (-T,T)\times B_{r}(x_0)$ and 
\[ \norm{D\Psi^H_i(t,x)-D \Psi^H_i(0,x_0)}\leq \frac 1{4C_1} \text{ for } (t,x)\in B_{2\sqrt{2}C_1\delta_0}(0,x_0),
\]
which is possible thanks to the equi-continuity.

If $\delta\leq \delta_0$ and $\ep:=\delta/(\sqrt{2}C_1)<\delta$,
then, according to the previous lemma, for each $i\in I$, the map 
\[
{\Psi^H_i}_{|U_i}:U_i\to  (-\delta,\delta)\times B_{\delta}(x_0) 
\]
is a diffeomorphism from an open set $(-\ep,\ep)\times B_{\ep}(x_0)\subseteq U_i\subseteq (-T,T)\times B_{r}(x_0)$. In particular, for each $t\in (-\ep,\ep)$ and $i\in I$, 
$\psi^H_{i,t}$ maps an open neighbourhood $B_{\ep}(x_0)\subseteq U_{i,t}\subseteq B_{r}(x_0)$ of $x_0$ diffeomorphically onto $B_{\delta}( x_0)$. 
It is clear that $\ep(\delta)\leq \ep(\delta')$ and $U_{i,t}(\delta)\subseteq U_{i,t}(\delta')$ whenever $\delta\leq \delta'$.

For later purposes (as required in step 3), note that it follows from the claim that, by choosing $\ep$ (and thus $t$) suficiently small, we can assume that the family of maps $\{\psi^H_{i,t}:U_{i,t}\to B_{\delta}(x_0)\mid |t|\leq \ep,\ i\in I\}$ is uniformly bi-Lipschitz. 
\medskip

 \noindent \textbf{Second step:  Definition of $\ep$ and $\delta$, \eqref{c} holds and $T_tf_i$ is smooth on $B_\delta(x_0)$}

\noindent  First, note that \eqref{c} is no longer local, as it involves the non-local Lax-Oleinik semigroup evolution\footnote{Even though \eqref{k} shows that the problem can be made local for $\delta$ small enough.}.
 
Fix $i\in I$ and denote $f:=f_i$. Theorem \ref{cc} ensures that the function $Tf$ is locally semiconcave on $(0,\infty)\times B_r(x_0)$. In particular, $Tf$ is differentiable almost everywhere and, again by Theorem \ref{cc}, if $(t,y)$ is a differentiability point of $Tf$, there is exactly one $x\in M$ such that
 \begin{align}
     T_tf(y)=f(x)+c_t(x,y). \label{j}
 \end{align}
 Moreover, $x\in I^-(y)$, and by Lemma \ref{ddd}, we have
 \begin{align*}
     d_y(T_tf)=\frac{\partial L}{\partial v}(y,\dot \gamma(t)), \ \frac{\partial L}{\partial v}(x,\dot \gamma(0))= d_xf,
 \end{align*}
 where $\gamma:[0,t]\to M$ is the unique (necessarily timelike) minimizing curve connecting $x$ with $y$. In particular, 
 \begin{align}
 \gamma(s)=\psi^H_s(x,d_xf) \text{ for } s\in [0,t], \label{eeee}
 \end{align}
 and, thanks to Remark \ref{lllll},
 \begin{align*}
     (y,d_y(T_tf))
     =(\gamma(t),\frac{\partial L}{\partial v}(\gamma(t),\dot \gamma(t)))
     ={\cal L}(\gamma(t),\dot \gamma(t))
     &={\cal L}(\phi_t(\gamma(0),\dot \gamma(0)))
     \\
     &=\psi^H_t(x,\frac{\partial L}{\partial v}(x,\dot \gamma(0)))
     \\
     &=\psi^H_t(x,d_{x}f).
 \end{align*}
Denote $\ep_0:=\ep(\delta_0)$ and set $\delta:=\min\{\delta_0,\ep_0/2,\ep_0/(2C(L+1))\}$, where $C(L+1)$ is as Definition \ref{e1} and $\delta_0$ is as in Step 1. Pick $\ep>0$ according to Step 1, then the first claim in the lemma holds by Step 1. Without loss of generality, we assume $\ep\leq T$.

We now prove \eqref{c}. Denote
\begin{align*}
    &\psi^{H,0}_{i,t}:U_{i,t}^0:=U_{i,t}(\delta_0)\to B_{\delta_0}(x_0),\ 0\leq t\leq \ep_0, \text{ and }
    \\[10pt]
    &\psi^H_{i,t}:U_{i,t}:=U_{i,t}(\delta)\to B_\delta(x_0),\ 0\leq t\leq \ep.
\end{align*}
Fix $0\leq t\leq \ep$. Let $y\in B_\delta(x_0)$ be a differentiability point of $Tf$, and let $x\in I^-(y)$ be the unique point such that
\begin{align*}
    T_tf(y)=f(x)+c_t(x,y).
\end{align*}
Using Lemma \ref{t} and the fact that, due to step 1, $\ep\leq \delta$, it follows that $d_h(x,y)\leq C(L+1)\ep \leq C(L+1)\delta\leq \ep_0/2$. Thus, 
\[
d_h(x,x_0)\leq d_h(x,y)+d_h(y,x_0)< \ep_0/2+\delta \leq \ep_0.
\]
Therefore, $x\in U_{i,t}^0$. Then we have
\begin{align*}
    \psi^{H,0}_{i,t}(x)=\pi\circ \psi^H_t(x,d_xf)=\pi(y,d_yTf)=y.
\end{align*}
But $\psi^H_{i,t}(U_{i,t})=B_\delta(x_0)$ and $\psi^{H,0}_{i,t}$ is an extension of $\psi^H_{i,t}$, so we must have $x\in U_{i,t}$. 
Thus, $x=(\psi^H_{i,t})^{-1}(y)$, and we have proved
\begin{align}
    (y,d_yT_tf)=\psi^H_t ((\psi^H_{i,t})^{-1}(y),d_{(\psi^H_{i,t})^{-1}(y)}f), \label{k}
\end{align}
where the latter expression is smooth on $B_\delta(x_0)$. So, $T_tf$ is locally Lipschitz continuous on $B_\delta(x_0)$ (since $\delta<r$), differentiable almost everywhere, and the (graph of the) differential can be extended smoothly to $B_\delta(x_0)$. Thus, $T_tf$ must be smooth. 

Since this argument applies to an arbitrary differentiability point $y\in B_\delta(x_0)$, a repetition of the argument shows that \eqref{k} holds everywhere on $B_\delta(x_0)$. By \eqref{eeee},
\[
T_tf(\psi^H_{i,t}(x))=f(x)+\int_0^t L(\psi^H_{i,s}(x,d_xf),\frac{d}{ds}\psi^H_{i,s}(x,d_xf))\, ds \text{ for all } x\in U_{i,t}.
\]

\noindent \textbf{Step 3: The family $\{T_tf_i\mid 0\leq t\leq \ep,\ i\in I\}$, is locally bounded on $B_\delta(x_0)$ in $C^2$}

\noindent Thanks to \eqref{c} and the fact that $U_{i,t}\subseteq B_r(x_0)$, going back with the chart $\phi$, this problem is again local. Thus, we can again assume that $M$ is an open subset of $\R^n$ and that the family $f_i:B_{2r}(x_0)\to \R$ is bounded in $C^2$.

By \eqref{k} it holds:
\begin{align}
    (\psi^H_{i,t}(x),D(T_tf_i)({\psi^H_{i,t}(x)}))=\psi^H_t(x, Df_i(x)) \text{ on } U_{i,t}. \label{eee}
\end{align}
To prove the $C^2$-boundedness, it clearly suffices to prove that the maps
\[
B_\delta(x_0)\to \R^{2n},\ y\mapsto (y,D(T_tf_i)(y)),
\]
are equi-Lipschitz.
However, as mentioned at the end of Step 1, the family of maps $(\psi^H_{i,t})^{-1}:B_\delta(x_0)\to U_{i,t}$, $i\in I$, $0\leq t\leq \ep$, is equi-Lipschitz. Moreover, also the family of maps
\begin{align*}
   B_{r}(x_0)\ni x\mapsto\psi^H_t(x,D f_i(x)),\ i\in I,\ 0\leq t\leq \ep,
\end{align*}
is equi-Lipschitz. Indeed, the family $(x\mapsto (x,Df_i(x)))_i$ is equi-Lipschitz on $B_{r}(x_0)$ due to the $C^2$-boundedness, and $\psi^H$ is Lipschitz on the compact set $[0,\ep]\times \overbar C$ because $\ep\leq T$. Thus, the claim follows from \eqref{eee}.
\end{proof}

\begin{proof}[Proof of Proposition \ref{o}]
We will only prove the result for $\psi$.
Thanks to Corollary \ref{s}, we already know that $T_\tau\psi_{t_0-\tau}$ is locally semiconcave on an open neighbourhood of $\gamma(t_0)$. Therefore, due to the well-known Lemma \ref{gggg}, it suffices to prove that $T_\tau\psi_{t_0-\tau}$ is locally semiconvex on an open neighbourhood of $\gamma(t_0)$.

We denote $x_0:=\gamma(t_0)$. By Lemma \ref{m} and Lemma \ref{h1}(ii), the map $\hat T\psi$ is locally semiconvex on a neighborhood of $(1-t_0,x_0)$. It is differentiable at $(1-t_0,x_0)$ due to the dual version of Lemma \ref{xx}. Furthermore, by Lemma \ref{ddd}, we have
$d_{x_0}\psi_{t_0}=\frac{\partial L}{\partial v}(x_0,\dot \gamma(t_0))$, so that
\[
d_{x_0}\psi_{t_0}(\dot \gamma(t_0))=2L(x_0,\dot \gamma(t_0))>0.
\]
In particular, $d_{x_0}(\hat T_{1-t_0}\psi)\in T_{x_0}^*M\backslash \C_{x_0}^*$.
Thus, we can apply Lemma \ref{r} in the appendix: 

There exists $r>0$ and a family of smooth equi-Lipschitz functions $\{f_{i,s}:M\to \R\mid i\in I_s,\ s\in (t_0-r,t_0+r)\}$ such that:

\begin{enumerate}[(i)]
    \item $(f_{i,s})$ is locally bounded in $C^2$.
    \item 
    $\psi_s= \sup_{i\in I_s} f_{i,s} \text{ on } B_{r}(x_0)$ for $s\in (t_0-r,t_0+r)$.
    \item 
    The set
\[
    \{(x,d_xf_{i,s})\mid x\in  B_{r}(x_0),\ s\in (t_0-r,t_0+r),\ i\in I_s\}
    \]
is relatively compact in $T^*M\backslash \C^*$.
\item 
For any $x\in B_{r}(x_0)$, $s\in (t_0-r,t_0+r)$ and $p\in \partial^- \psi_s(x)$, we find $i\in I_s$ with $f_{i,s}(x)=\psi_s(x)$ and $d_xf_{i,s}=p$.
\end{enumerate} 
Note that $\psi_s=\hat T_{1-s}\psi$, so that (as given by Lemma \ref{r}) if $1-s\in (1-t_0-r,1-t_0+r)$, then $s\in (t_0-r,t_0+r)$.

Due to the Lipschitz property, as well as properties (i) and (iii), this family meets the assumptions of the previous lemma. This ensures the existence of $0<\ep,\delta<r$ such that: 
Given $0\leq t\leq \ep$, $s\in (t_0-r,t_0+r)$ and $i\in I_s$, the map
\begin{align*}
    \psi^H_{i,s,t}:U_{i,s,t}\to B_\delta(x_0),\ x\mapsto \pi \circ \psi^H_t(x,d_xf_{i,s}),
\end{align*}
is a diffeomorphism from an open set $U_{i,s,t}$ containing $B_\ep(x_0)$ and satisfies
\begin{align*}
    T_tf_{i,s}(\psi^H_{i,s,t}(x))=f_{i,s}(x)+c_t(x,\psi^H_{i,s,t}(x)),\ x\in U_{i,s,t}.
\end{align*}
Moreover, the family $\{T_tf_{i,s}\mid s\in (t_0-r,t_0+r),\ i\in I_s,\ 0\leq t\leq \ep\}$ is locally bounded in $C^2(B_\delta(x_0))$.
To conclude the proof, it suffices to show that, for some $\delta'<\delta$ and sufficiently small $\tau<\ep$, $T_\tau\psi_{t_0-\tau}$ is finite on $B_{\delta'}(x_0)$ and
\begin{align}
\sup_{i\in I_{t_0-\tau}}  T_\tau f_{i,t_0-\tau} = T_\tau\psi_{t_0-\tau} \text{ on } B_{\delta'}(x_0). \label{mmmmm}
\end{align}
Indeed, since the family $(T_\tau f_{i,t_0-\tau})_{i\in I_{t_0-\tau}}$ is locally bounded in $C^2(B_\delta(x_0))$, it follows that $ T_\tau\psi_{t_0-\tau}$ is locally semiconvex on $B_{\delta'}(x_0)$ (see Remark \ref{pppp}). This would finish the proof.
Let us now verify \eqref{mmmmm}.
\medskip

\noindent \textbf{Claim:} There exists $\delta'<\delta$ such that, for sufficiently small $\tau>0$, the function $T_\tau\psi_{t_0-\tau}$ is finite on $B_{\delta'}(x_0)$ and we have 
\[
 T_\tau\psi_{t_0-\tau}(y)= \min\{\psi_{t_0-\tau}(x)+c_\tau(x,y)\mid x\in B_\ep(x_0)\cap I^-(y)\} \text{ for all } y\in B_{\delta'(x_0)}.
\]

\noindent \textbf{Proof:} By the local semiconvexity of $\hat T\psi$ near $(1-t_0,x_0)$, we find $\ep'<\ep$ such that the map $\hat T\psi$ is Lipschitz on $[1-t_0-\ep',1-t_0+\ep']\times \overbar B_{\ep'}(x_0)$. Set 
\[
M:=\inf\{\psi_{t_0-\tau}(x)\mid (\tau,x)\in [-\ep',\ep']\times \overbar B_{\ep'}(x_0)\}.
\]
Now, let $\delta''<\min\{\ep',\delta\}$ be fixed. Moreover, let $\tau<\ep'$.

If $y\in B_{\delta''}(x_0)$ and $x\in J^-(y)$ with $x\notin B_{\ep'}(x_0)$, pick a minimizer $\tilde \gamma:[0,\tau]\to M$ connecting $x$ with $y$, and pick $t\in [0,\tau)$ with $\tilde \gamma(t)\in \partial B_{\ep'}(x_0)$. Then, using the semigroup property and \eqref{splitting}, we have
\begin{align*}
\psi_{t_0-\tau}(x)+c_\tau(x,y)&\geq \psi_{t_0-\tau+t}(\tilde \gamma(t))-c_{t}(x,\tilde \gamma(t))+c_{t}(x,\tilde \gamma(t))+ c_{\tau-t}(\tilde \gamma(t),y)
\\[10pt]
&
= \psi_{t_0-\tau+t}(\tilde \gamma(t))+c_{\tau-t}(\tilde \gamma(t),y).
\\[10pt]
& \geq \psi_{t_0-\tau+t}(\tilde \gamma(t))+\frac{d_h(\tilde \gamma(t),y)^2}{4(\tau-t)}
\\[10pt]
& \geq M +\frac{(\ep'-\delta'')^2}{4\tau}.
\end{align*}
The latter term is independent of $x$ and $y$ and converges to $\infty$ as $\tau\to 0$. Thus,
\[
\inf\{\psi_{t_0-\tau}(x)+c_\tau(x,y)\mid y\in B_{\delta''}(x_0),\ x\notin B_{\ep'}(x_0)\} \xrightarrow{\tau\to 0} \infty.
\]
However, $T_\tau\psi_{t_0-\tau}\leq \psi_{t_0-\tau}$ is uniformly bounded from above on $B_{\delta''}(x_0)$ for $|\tau|\leq \ep'$ since $\hat T\psi$ is continuous on $[1-t_0+\ep',1-t_0+\ep']\times \overbar B_{\ep'}(x_0)$ (recall that $\delta''<\ep'$). Thus, for a sufficiently small $\tau\leq \ep'$, which we now fix, we have 
\[
T_\tau\psi_{t_0-\tau}(y)< \inf\{\psi_{t_0-\tau}(x)+c_r(x,y)\mid x\notin  B_{\ep'}(x_0)\} \text{ for all } y\in B_{\delta''}(x_0).
\]
Since $\tau\leq \ep'$, $\psi_{t_0-\tau}$ is continuous on $\overbar B_{\ep'}(x_0)$, so that a simple compactness argument shows that the infimum is attained at a point $x\in B_{\ep'}(x_0)$. In particular, $T_\tau\psi_{t_0-\tau}$ is finite on $B_{\delta''}(x_0)$. Moreover, we can redo the proof of Proposition \ref{i1} (which uses the Lipschitz property of $\psi_{t_0-\tau}$ on $B_{\ep'}(x_0)$) to obtain that $x\notin \partial J^-(y)\backslash \{y\}$. It remains to prove that there exists $\delta'<\delta''$ such that, for any $y\in B_{\delta'}(x_0)$, we have $x\neq y$. Indeed, assume that there exists a sequence $y_k\to x_0$ such that $T_{\tau}\psi_{t_0-\tau}(y_k)=\psi_{t_0-\tau}(y_k)$. Since $T_{\tau}\psi_{t_0-\tau}$ is locally semiconcave on a neighbourhoud of $x_0$ (Corollary \ref{s}) and $\psi_{t_0-\tau}$ is continuous at $x_0$, it follows from Lemma \ref{h1} that
\begin{align*}
\varphi(\gamma(0))+c_{t_0}(\gamma(0),\gamma(t_0))&=T_{\tau}\psi_{t_0-\tau}(x_0)=\psi_{t_0-\tau}(x_0)
\\
&=\varphi(\gamma(0))+c_{t_0-\tau}(\gamma(0),\gamma(t_0-\tau)).
\end{align*}
This is a contradiction and proves the claim.
\hfill \checkmark

Now, let $\delta'<\delta$ be as in the claim and $0<\tau<\min\{\ep,r\}$ sufficiently small to make the claim applicable. From the claim and the fact that $f_{i,t_0-\tau}\leq \psi_{t_0-\tau}$ on $B_{r}(x_0)\supseteq B_\ep(x_0)$ (see (ii)), it follows that 
\[
\sup_{i\in I_{t_0-\tau}} T_\tau f_{i,t_0-\tau}\leq T_\tau\psi_{t_0-\tau} \text{ on $B_{\delta'}(x_0)$.}
\]
We aim to prove that equality holds.

 Let $y\in B_{\delta'(x_0)}$ and $x\in B_{\ep}(x_0)\cap I^-(y)$ such that the identity
\[
T_\tau\psi_{t_0-\tau}(y)=\psi_{t_0-\tau}(x)+c_\tau(x,y)
\]
holds. Let $\tilde \gamma:[0,\tau]\to M$ be a minimizer connecting $x$ to $y$. Since $x\in I^-(y)$, $\tilde \gamma$ must be timelike, so that $v:=\frac{\partial L}{\partial v}(x,\dot {\tilde \gamma}(0))\in \partial^-\psi_{t_0-\tau}(x)$ thanks to Lemma \ref{ddd}. Since $x\in B_\ep(x_0)\subseteq B_r(x_0)$, it follows from (iv) that there is $i\in I_{t_0-\tau}$ with $v=d_xf_{i,t_0-\tau}$ and $f_{i,t_0-\tau}(x)=\psi_{t_0-\tau}(x)$. As $\tilde \gamma$ is a minimizer, its speed curve is an orbit of the Euler-Lagrange flow, this means that $\tilde \gamma(t)=\pi \circ \psi^H_{t}(x,d_xf_{i,t_0-\tau})$. In particular, since $\tau\leq \ep$, the previous lemma gives 
\[
 T_\tau f_{i,t_0-\tau}(y)= f_{i,t_0-\tau}(x)+c_\tau(x,y) = \psi_{t_0-\tau}(x)+c_\tau(x,y)= T_\tau\psi_{t_0-\tau}(y).
\]
This finally proves \eqref{mmmmm}.
\end{proof}

\begin{proof}[Proof of Theorem \ref{n}]
     For each $\gamma\in A$, fix $\tau=\tau(\gamma)<(t-s)/2$ such that the statement of Proposition \ref{o} is applicable, i.e. the function 
        $\hat T_{\tau(\gamma)} \varphi_{s+\tau(\gamma)}$ is $C_{loc}^{1,1}$ on an open neighbourhood $U_\gamma$ of $\gamma(s)$, and the function $T_{\tau(\gamma)}\psi_{t-\tau(\gamma)}$ is $C_{loc}^{1,1}$ on an open neighbourhood $V_\gamma$ of $\gamma(t)$. 
        Let 
        \begin{align*}
            U:=\bigcup_{\gamma\in A} U_\gamma \text{ and } V:=\bigcup_{\gamma\in A} V_\gamma.
        \end{align*}
       By definition, the family $(U_\gamma)_{\gamma \in A}$ forms an open cover of the manifold $U$, so we can choose a smooth partition of unity subordinate to this open cover. That is, there exist smooth functions $(\rho_\gamma:U\to [0,1])_{\gamma\in A}$ such that 
       $\supp(\rho_\gamma)\subseteq U_\gamma$ (the support is taken w.r.t. the manifold $U$) and such that every $x\in U$ has a neighbourhood intersecting only finitely many $\supp(\rho_\gamma)$. We then define
    \begin{align*}
        U\ni x\mapsto \Phi_{s}(x):=\sum_{\gamma \in A} \rho_\gamma(x) \hat T_{\tau(\gamma)}\varphi_{s+\tau(\gamma)}(x),
    \end{align*}
    where, of course, $0\cdot (\pm \infty):=0$.
    Locally this is a finite sum of functions that are all $C_{loc}^{1,1}$, so also the sum itself is $C_{loc}^{1,1}$.
Similarly, we define $\Psi_{t}:V\to \R$ by
\[
\Psi_{t}(y):=\sum_{\gamma \in A} \eta_\gamma(y) T_{\tau(\gamma)}\psi_{t-\tau(\gamma)}(y),
\]
where $(\eta_\gamma)_\gamma$ is another partition of unity subordinate to the open cover $(V_\gamma)_\gamma$ of $V$. In particular, $\Psi_{t}$ is also $C_{loc}^{1,1}$. 

We need to prove that the pair $(\Phi_{s},\Psi_{t})$ is $(c_{t-s},\pi_{s,t})$-calibrated. More precisely, we mean that the functions obtained by extending $\Phi_{s}$ and $\Psi_{t}$ to the whole of $M$ by setting $\Phi_{s}(x)=\infty$ for $x\notin U$, and $\Psi_{t}(y)=-\infty$ for $y\notin V$, are $(c_{t-s},\pi_{s,t})$-calibrated. However, with this definition, it is clear that we have to check the subsolution property only for $(x,y)\in U\times V$.

 Thus, let $(x,y)\in U\times V$ and $\gamma,\gamma' \in A$ be arbitrary. Thanks to Lemma \ref{h1}(iii), we have
        \[
       T_{\tau(\gamma)}\psi_{t-\tau(\gamma)}(y)- \hat T_{\tau(\gamma')}\varphi_{s+\tau(\gamma')}(x)
       \leq
       c_{t-s}(x,y).
        \]
        Therefore,
        \begin{align*}
        \Psi_{t}(y)-\Phi_{s}(x) &= \sum_{\gamma \in A}  \eta_\gamma(y)T_{\tau(\gamma)}\psi_{t-\tau(\gamma)}(y)- \sum_{\gamma' \in A} \rho_{\gamma'}(x) \hat T_{\tau(\gamma')}\varphi_{s+\tau(\gamma')}(x)
        \\[10pt]
        &= \sum_{\gamma,\gamma' \in A} \rho_{\gamma'}(x) \eta_\gamma(y) \left( T_{\tau(\gamma)}\psi_{t-\tau(\gamma)}(y)-  \hat T_{\tau(\gamma')}\varphi_{s+\tau(\gamma')}(x)\right)
        \\[10pt]
        &\leq
        \sum_{\gamma,\gamma' \in A} \rho_{\gamma'}(x) \eta_\gamma(y) c_{t-s}(x,y)=c_{t-s}(x,y).
        \end{align*}
        Thus, $(\Phi_s,\Psi_t)$ is a $c_{t-s}$-subsolution. The $(c_{t-s},\pi_{s,t})$-calibration then follows from Lemma \ref{jjjjj} and the last claim of the theorem, namely that $\Phi_s=\varphi_s$ on $A_s$ and $\Psi_t=\psi_t$ on $A_t$. This, however, follows immediately from Lemma \ref{h1}(i),(ii) and the definitions of $\Phi_s$ and $\Psi_t$. 
\end{proof}

\section{Appendix}
\subsection{Semiconcave functions and boundedness in $C^2$}\label{semiconcavity}

In this section, we review some basic properties of semiconcavity and semiconvexity, as well as other related definitions. For proofs and additional results, we refer the reader to Appendix A in \cite{Fathi/Figalli}.

\begin{definition}\rm
    \begin{enumerate}[(a)]
\item
    A function $f:U\to \overbar \R$ defined on an open set $U\subseteq \R^n$ is said to be \emph{super-differentiable} (resp. \emph{sub-differentiable}) at $x\in U$ with \emph{super-differential} (resp. \emph{sub-differential}) $p\in (\R^n)^*$ if $f(x)\in \R$ and
    \begin{align*}
    &f(y)\leq f(x)+ \langle p,y-x\rangle+o(|y-x|) 
    \\[10pt]
    (\text{resp. }  &f(y)\geq f(x)+ \langle p,y-x\rangle+o(|y-x|)).
    \end{align*}
    \item A function $f:U\to \overbar \R$ defined on an open subset $U\subseteq N$ of a smooth manifold $N$ is said to be \emph{super-differentiable} at $x\in U$ with \emph{super-differential} $p\in T_x^*N$ if $f(x)\in \R$ and if, for one (or any) chart $(\phi,V)$ of $U$ around $x$, the function $f\circ \phi^{-1}$ is super-differentiable at $\phi(x)$ with super-differential $p\circ d_{\phi(x)}\phi^{-1}$.
We denote by $\partial^+f(x)$ the set of all super-differentials at $x$.

A similar definition holds for sub-differentiable functions, whose sub-differentials at $x$ are denoted by $\partial^-f(x)$.
    \end{enumerate}
\end{definition}

\begin{definition}\rm  \label{concave}
\begin{enumerate}[(a)]
\item A function $f:U\to \R$ defined on an open set $U\subseteq \R^n$ is said to be \emph{$K$-semiconcave} if, for all $x\in U$, there exists $p\in (\R^n)^*$ such that
    \begin{align}
         f(y)\leq f(x)+\langle p,y-x\rangle +K|y-x|^2 \ \forall y\in U. \label{ssss}
    \end{align}
		\item A family of functions $f_i:U\to \R$, $i\in I$, defined on an open set $U\subseteq \R^n$ is said to be \emph{uniformly semiconcave} if the constant $K$ in definition (a) can be chosen uniformly for all functions $f_i$.	
\item
    A function $f:U\to \R$ defined on an open set $U\subseteq \R^n$ is said to be \emph{locally semiconcave} if, for every $x\in U$, there exists an open neighborhood $V\subseteq U$ around $x$ such that $f$ is $K$-semiconcave on $U$ for some $K>0$.

    \item A function $f:U\to \R$ defined on an open subset $U\subseteq N$ of a smooth manifold $N$ is said to be \emph{locally semiconcave} if, for each chart $(\phi,V)$ of $U$, the function $f\circ \phi^{-1}$ is locally semiconcave.

    \item  A family of functions $f_i:U\to \R$, $i\in I$, defined on an open subset $U\subseteq N$ of a smooth manifold $N$ is said to be \emph{uniformly locally semiconcave} if, for each $x\in U$, there exists a chart $(\phi,V)$ of $U$ around $x$ such that the family of functions $f_i\circ \phi^{-1}$, $i\in I$, is uniformly semiconcave.

		\item All of the above definitions also apply to semiconvexity by replacing $f$ with $-f$.
    \end{enumerate}
\end{definition}

\begin{remark}\rm
\begin{enumerate}[(a)]
\item 
It is shown in \cite{Fathi/Figalli}, Corollary A14, that the infimum of a uniformly locally semiconcave family of functions is locally semiconcave, provided the infimum is finite.

An analogous result holds for semiconvex functions, replacing the infimum with the supremum.
\item Note that, in the situation of part (a) of the above definition, inequality \eqref{ssss} is only required to hold for some $p\in (\R^n)^*$, and necessarily $p\in \partial^+f(x)$. However, the inequality holds for any other $q\in \partial^+f(x)$ as well (for the same $K$). 
\item It is well-known that locally semiconcave and locally semiconvex functions are locally Lipschitz, see Lemma A5 in \cite{Fathi/Figalli}.
\end{enumerate}
\end{remark}

\begin{definition}\rm
\begin{enumerate}[(a)]
    \item A family of $C^2$ functions $f_i:U\to \R$, $i\in I$, defined on an open set $U\subseteq \R^n$ is said to be \emph{bounded in $C^2$} if 
    \[
    \sup_{i\in I,\ x\in U} \norm{D^2f_i(x)}<\infty.
    \]
    
    \item 
    A family of $C^2$ functions $f_i:U\to \R$, $i\in I$, defined on an open subset $U\subseteq N$ of a smooth manifold $N$ is said to be \emph{locally bounded in $C^2$} if, for each $x\in U$, there exists a chart $(\phi,V)$ around $x$ such that the family of maps $f_i\circ \phi^{-1}$, $i\in I$, is bounded in $C^2$.
\end{enumerate}
\end{definition}

\begin{remark}\rm\label{pppp}
\begin{enumerate}[(a)]
\item
Note that, similar to the situation in \cite{Fathi/Figalli}, in part (b) of the above definition, we require the property to hold for only one chart. It is not necessarily true that the property holds for all charts.
We refer to Example A12 in \cite{Fathi/Figalli}, which justifies the definition of uniform local semiconcavity.

\item
It is clear from the definition that a family of functions that is locally bounded in $C^2$ is uniformly locally semiconcave and semiconvex. In particular, the finite supremum (resp.\ infimum) of a family of functions that is locally bounded in $C^2$ is locally semiconvex (resp.\ semiconcave).
\end{enumerate}
\end{remark}

\begin{definition}\rm \label{eeeee}
\begin{enumerate}[(a)]
	\item
For an open subset $U\subseteq \R^n$, we define $C^{1,1}(U)$ as the set of all $C^1$ functions with Lipschitz derivative, i.e.\ the mapping
\[
U\to (\R^n)^*\cong \R^n,\ x\mapsto Df(x),
\]
is Lipschitz. With obvious modifications as in Definition \ref{concave}, we define $C^{1,1}_{loc}(U)$.
\item
If $U\subseteq N$ is an open subset of a smooth manifold $N$, we define $C_{loc}^{1,1}(U)$ as the set of $C^1$ functions $f:U\to \R$ such that, for any chart $(\phi,V)$ of $U$, we have $f\circ \phi^{-1}\in C_{loc}^{1,1}(\phi(V))$.
\end{enumerate}
\end{definition}

\begin{lemma}\label{gggg}
	A function $f:U\to \R$ defined on an open set of a smooth manifold belongs to $C_{loc}^{1,1}(U)$ if and only if it is locally semiconcave and locally semiconvex.
\end{lemma}
\begin{proof}
    See \cite{Fathi}.
\end{proof}

The following theorem (in a slightly different setting) is taken from \cite{Bernard}, Proposition 10. We provide the proof here due to the minor differences. 

\begin{lemma}\label{r}
Suppose that $M$ is (as in the paper) a globally hyperbolic spacetime.

    Let $f:(0,\infty)\times M\to \overbar \R$ be a function, which is locally semiconvex in a neighbourhood of $(t_0,x_0)$ and differentiable at $(t_0,x_0)$ with $d_{x_0}f(t_0,\cdot)\in T_{x_0}^*M\backslash \C_{x_0}^*$. 
    
    Then there exists $r>0$ and a family of smooth equi-Lipschitz functions $f_{i,s}:M\to \R$, $i\in I_s$, $s\in (t_0-r,t_0+r)$, such that 
    \begin{enumerate}[(i)]
        \item $(f_{i,s})$ is locally bounded in $C^2$.
        \item 
    $f(s,\cdot)= \sup_{i\in I_s} f_{i,s} \text{ on } B_{r}(x_0)$
    \item 
    The set
    \[
    \{(x,d_xf_{i,s})\mid x\in  B_r(x_0),\ s\in (t_0-r,t_0+r),\ i\in I_s\}
    \]
    is relatively compact in $T^*M\backslash \C^*$.
    \item 
     For any $x\in B_{r}(x_0)$, $s\in (t_0-r,t_0+r)$ and $p\in \partial^-f_s(x)$, we find $i\in I_s$ with $f_{i,s}(x)=f(s,x)$ and $d_xf_{i,s}=p$.
    \end{enumerate}
\end{lemma}
\begin{proof}
     
    Recall that locally semiconvex functions are locally Lipschitz. Hence, from this fact and by assumption, there exist constants $r',K>0$, as well as a chart $(U,\phi)$ centered at $x_0$ with $\phi(U)=B_3(0)$  such that:
     \begin{align*}
        \forall s\in (t_0-r',t_0+r'):\ f_s\circ \phi^{-1} \text{ is $K$-convex and $K$-Lipschitz}.
     \end{align*}
     Here, $f_s:=f(s,\cdot)$. Let us consider, for any $(s,x)\in (t_0-r',t_0+r')\times \phi^{-1}(B_1(0))$ and $p\in \partial^-f_s(x)$, the smooth function $f_{x,p,s}:M\to \R$ defined by
     \begin{align}
        &f_{x,p,s}\circ \phi^{-1}:B_3(0)\to \R, \nonumber
        \\[10pt]
        &z\mapsto \rho(z) \left(f_s(x)+p\circ d_{\phi(x)}\phi^{-1}(z-\phi(x))+K|z-\phi(x)|^2\right), \label{rrrr}
     \end{align}
    where $\rho:\R^n\to [0,1]$ is a smooth function satisfying $\rho(z)=1$ for $|z|\leq 1$ and $\rho(z)=0$ for $|z|\geq 2$. We set $f_{x,p,s}\equiv 0$ on $M\backslash U$.

The family is equi-Lipschitz because any sub-differential of $f_s\circ \phi^{-1}$, $s\in (t_0-r',t_0+r'),$ is bounded by $K$. For the same reason, the family $(f_{x,p,s}\circ \phi^{-1})$ is bounded in $C^2(B_3(0))$, so that $(f_{x,p,s})$ is locally bounded in $C^2$, proving (i) if $r<r'$. In addition, due to the $K$-convexity of $f_s\circ \phi^{-1}$, we obtain, for $s \in (t_0-r',t_0+r')$,
  \[
  f_s\geq \sup\{f_{x,p,s}\mid x\in B_1(0),\ p\in \partial^-f_s(x)\} \text{ on } \phi^{-1}(B_1(0)).
\]
Additionally, if $x\in \phi^{-1}(B_1(0))$, then $f_s(x)=f_{x,p,s}(x)$ for any $p\in \partial^-f_s(x)$ (recall that, since $f_s$ is locally semiconvex, every point admits a sub-differential). Thus, for $s \in (t_0-r',t_0+r')$, we proved that
\[
  f_s= \sup\{f_{x,p,s}\mid x\in B_1(0),\ p\in \partial^- f_s(x)\} \text{ on } \phi^{-1}(B_1(0)),
\]
proving (ii) for $r<r'$.

Next, since $f$ is locally semiconvex and differentiable at $(t_0,x_0)$, it is well-known that the set of sub-differentials $\partial^-f_s$ converge to $d_{x_0}f_{t_0}$ as $(s,x)\to (t_0,x_0)$ (\cite{Philippis}, Appendix A). Thus, from \eqref{rrrr}, it follows easily that $d_yf_{x,p,s}$ converges to $d_{x_0}f_{t_0}$ as $s\to t_0$ and $x,y\to x_0$. Since 
       $d_{x_0}f_{t_0}\in T^*_{x_0}M\backslash \C_{x_0}^*$, we can choose $r'>r>0$ with $B_{r}(x_0)\subseteq \phi^{-1}(B_1(0))$ such that the set
   \[
   \{(y,d_yf_{x,p,s})\mid y\in B_r(x_0),\ x\in B_r(x_0),\ s\in (t_0-r,t_0+r),\ p\in \partial^-f_s(x)\}
   \]
   is precompact in $T^*M\backslash \C^*$.  Then the family $(f_{x,p,s})$, $(s,x)\in (t_0-r,t_0+r)\times B_r(x_0)$, $p\in \partial^-f_s(x)$, obviously satisfies (i)-(iii). Since $f_s(x)=f_{x,p,s}(x)$ and $d_xf_{x,p,s}=p$ for any $x\in \phi^{-1}(B_1(0))$ and $p\in \partial^-f_s(x)$, also (iv) holds.
\end{proof}

\subsection{Proof of Lemma \ref{h1}}\label{proofs}

\begin{proof}[Proof of Lemma \ref{h1}]
\begin{enumerate}[(i)]
\item Let us first prove that $\gamma$ is $T\varphi$-calibrated. Since $\varphi(\gamma(0))\in \R$ and by Lemma \ref{w}, it suffices to check that $T_1\varphi(\gamma(1))=\varphi(\gamma(0))+c(\gamma(0),\gamma(1))$. However,
the inequality $\leq$ holds by definition. Moreover, since $(\varphi,\psi)$ is a $c$-subsolution, we have $\psi\leq T_1\varphi$, and
$\psi(\gamma(1))-\varphi(\gamma(0))=c(\gamma(0),\gamma(1)$
because $\gamma \in A$. This proves the required identity.

Now, using the semigroup property, Lemma \ref{ineq} and the fact that $\gamma$ is $T\varphi$-calibrated, we have 
\begin{align*}
&\hat T_r\varphi_{s+\tau}=\hat T_rT_r \varphi_{s+\tau-r}\leq \varphi_{s+\tau-r}
\\
\Rightarrow \ & \hat T_r\varphi_{s+\tau}(\gamma(s+\tau-r)\leq \varphi(\gamma(0))+c_{s+\tau-r}(\gamma(0),\gamma(s+\tau-r))
\end{align*}

Next, applying the definition of the Lax-Oleinik semigroup, the fact that $\gamma$ is $T\varphi$-calibrated and that $\gamma$ is a minimizer, we obtain
\begin{align*}
\hat T_r\varphi_{s+\tau}(\gamma(s+\tau-r))
&\geq \varphi_{s+\tau}(\gamma(s+\tau))-c_r(\gamma(s+\tau-r),\gamma(s+\tau))
\\[10pt]
&=\varphi(\gamma(0))+c_{s+\tau-r}(\gamma(0),\gamma(s+\tau-r)).
\end{align*}
Both inequalities combined prove part (i).
\item 
Let us first prove that, for any $t'\in [0,1]$, we have
\[
\psi_{t'}(\gamma({t'}))=\varphi(\gamma(0))+c_{t'}(\gamma(0),\gamma({t'})).
\]
Since $\psi_{t'}=\hat T_{1-{t'}}\psi$ and $\psi\leq {T}_1\varphi$, it follows from the semigroup property, Lemma \ref{ineq} and the $T\varphi$-calibration of $\gamma$ that
\[
\psi_{t'}(\gamma({t'}))\leq T_{t'}\varphi(\gamma({t'}))=\varphi(\gamma(0))+c_{t'}(\gamma(0),\gamma({t'})).
\]
On the other hand, since $\gamma \in A$,
\begin{align*}
\psi_{t'}(\gamma({t'}))
&\geq \psi(\gamma(1))-c_{1-{t'}}(\gamma({t'}),\gamma(1))
\\[10pt]
&=\varphi(\gamma(0))+c(\gamma(0),\gamma(1))-c_{1-{t'}}(\gamma({t'}),\gamma(1))
\\[10pt]
&=\varphi(\gamma(0))+c_{t'}(\gamma(0),\gamma({t'})).
\end{align*}
This proves our claim. Now, by the semigroup property and Lemma \ref{ineq}, 
\begin{align*}
T_r\psi_{t-\tau}(\gamma(t-\tau+r))
&= T_r\hat T_r \psi_{t-\tau+r}(\gamma(t-\tau+r)) \\[10pt]
&\geq \psi_{t-\tau+r}(\gamma(t-\tau+r))
\\[10pt]
&=\varphi(\gamma(0))+c_{t-\tau+r}(\gamma(0),\gamma(t-\tau+r)).
\end{align*}
On the other hand, by definition of the Lax-Oleinik semigroup, we have
\begin{align*}
    T_r\psi_{t-\tau}(\gamma(t-\tau+r)) 
    &\leq 
    \psi_{t-\tau}((\gamma(t-\tau))+c_r(\gamma(t-\tau),\gamma(t-\tau+r))
    \\[10pt]
    &=\varphi(\gamma(0))+c_{t-\tau}(\gamma(0),\gamma(t-\tau))+c_r(\gamma(t-\tau),\gamma(t-\tau+r))
    \\[10pt]
    &=
    \varphi(\gamma(0))+c_{t-\tau+r}(\gamma(0),\gamma(t-\tau+r)).
\end{align*}
\item 
Since $t-s\geq \tau_1+\tau_2$, we have
    \begin{align*}
        T_{t-s}(\hat T_{\tau_2}\varphi_{s+\tau_2})=T_{\tau_1} \circ T_{t-s-\tau_1-\tau_2} \circ T_{\tau_2} \circ \hat T_{\tau_2} (\varphi_{s+\tau_2})
        \geq T_{\tau_1} \circ T_{t-s-\tau_1-\tau_2}(\varphi_{s+\tau_2})
    \end{align*}
    Moreover, since $\psi\leq T_1\varphi$, also $\hat T_1\psi\leq \hat T_1T_1\varphi\leq \varphi$. Hence,
    \begin{align*}
        T_{t-s-\tau_1-\tau_2}(\varphi_{s+\tau_2})&\geq (T_{t-s-\tau_1-\tau_2} \circ T_{s+\tau_2}) \circ (\hat T_{t-\tau_1} \circ \hat T_{1-t+\tau_1})( \psi)
        \\
        &= T_{t-\tau_1} \circ \hat T_{t-\tau_1} \circ \psi_{t-\tau_1} 
        \\
        &\geq \psi_{t-\tau_1}.
    \end{align*}
    Both equations combined show that $T_{t-s}(\hat T_{\tau_2} \varphi_{s+\tau_2})\geq T_{\tau_1} \psi_{t-\tau_1}$. But this means that the pair $(\hat T_{\tau_2}\varphi_{s+\tau_2}, T_{\tau_1}\psi_{t-\tau_1})$ is a $c_{t-s}$-subsolution. 
    
    The other two claims follow from part (i) and (ii).
\end{enumerate}
\end{proof}

\subsection{Review of Lorentzian geometry}

\begin{definition}\rm
A \emph{spacetime} $(M,g)$ consists of a smooth and connected manifold (Hausdorff and second-countable) and a symmetric $(0,2)$-tensor field $g$ of constant signature $(-,+,...,+)$. Moreover, $M$ is time oriented, i.e.\ there exists a smooth global vector field $X:M\to TM$ such that, for each $x\in M$, $g_x(X(x),X(x))<0$.
\end{definition}

\begin{definition}\rm
A vector $v\in T_xM$ is called \emph{timelike, spacelike, or lightlike} if 
\begin{align*}
g_x(v,v)<0,\ g_x(v,v)>0 \text{ or } g_x(v,v)=0 \text{ and } v\neq 0.
\end{align*}
Timelike and lightlike vectors as well as the null-vector are referred to as \emph{causal}. A causal vector $v$ is said to be \emph{future-directed} (resp.\ \emph{past-directed}) if $g_x(v,X(x))\leq 0$ (resp.\ $\geq 0$).
For $x\in M$, we denote by $\C_x\subseteq T_xM$ the set of all future-directed causal vectors. Then, $\operatorname{int}(\C_x)$ consists of all future-directed timelike vectors, while $\partial \C_x$ consists of the future-directed lightlike vectors and $0$. We also set $\C:=\{(x,v)\in TM\mid v\in \C_x\}$.
\end{definition}

\begin{definition}\rm
We say that a locally absolutely continuous curve $\gamma:I\to M$, with $I\subseteq \R$ an interval, is \emph{future-directed causal}  if $\dot \gamma(t)$ is future-directed causal for almost every $t$.

A $C^1$-curve $\gamma:I\to M$, with $I\subseteq \R$ an interval, is called future-directed timelike if $\dot \gamma(t)$ is future-directed timelike for every $t$.
\end{definition}

\begin{definition}\rm
Two points $x,y\in M$ are said to be \emph{causally related} (resp.\ \emph{chronologically related}) if there exists a future-directed causal (resp. timelike) $C^1$-curve connecting them. In this case, we write $x<y$ (resp. $x\ll y$). We write $x\leq y$ if $x=y$ or $x<y$.
The relations $>,\gg$ and $\geq$ are defined analogously.

 We define the \emph{chronological future/past} and \emph{causal future/past} of a point $x\in M$ as follows:
\begin{align*}
&I^+(x):=\{y\in M\mid y\gg x\},
\\
&I^-(x):=\{y\in M\mid y\ll x\},
\\
&J^+(x):=\{y\in M\mid y\geq x\},
\\
&J^-(x):=\{y\in M\mid y \leq x\}.
\end{align*}
We also set 
\begin{align*}
J^+:=\{(x,y)\in M\mid y\in J^+(x)\} \text{ and } I^+:=\{(x,y)\in M\mid y\in I^+(x)\}.
\end{align*}
\end{definition}

\begin{remark}\rm
    An equivalent definition can be given by also considering absolutely continuous causal curves. The equivalence of these definitions is established in \cite{Minguzzi}, Theorem 2.9.
\end{remark}

\begin{definition}\rm
The \emph{length} of an absolutely continuous future-directed causal curve $\gamma:[a,b]\to M$ is defined by
\begin{align*}
L(\gamma):=\int_a^b \sqrt{-g(\dot \gamma(t),\dot \gamma(t))}\, dt \in [0,\infty).
\end{align*}
The \emph{(Lorentzian) distance function} or \emph{time seperation} is the function $d:M\times M\to \R\cup\{\infty\}$, defined such that for $x<y$, the distance $d(x,y)$ is the supremum of $L(\gamma)$ over all absolutely continuous (or $C^1$, see \cite{Minguzzi}) future-directed causal curves connecting $x$ with $y$, and such that $d(x,y)=0$ if $x\nless y$.

A curve $\gamma$ connecting $x$ with $y>x$ is said to be \emph{(length) maximizing} if $L(\gamma)=d(x,y)$.
\end{definition} 

In this paper we focus on globally hyperbolic spacetimes, which are defined as follows:

\begin{definition}\rm
A spacetime $M$ is said to be \emph{globally hyperbolic} if there is no causal loop (i.e. no closed future-directed causal curve), and if for any $x,y\in M$, the intersection $J^+(x)\cap J^-(y)$ is compact.
\end{definition}

\begin{proposition}
    Let $M$ be a globally hyperbolic spacetime. Then the distance function is real-valued and continuous. Moreover, for any points $x<y$, there exists a length maximizing geodesic connecting $x$ and $y$.
Additionally, the set $J^+$ is closed.
\end{proposition}
\begin{proof}
    See \cite{ONeill}, Chapter 14, Proposition 19 and Lemma 22. 
\end{proof}

 \begin{remark}\rm
 It is often helpful to fix an arbitrary complete Riemannian metric on $M$, which we will denote by $h$.
Recall that, according to the Whitney embedding theorem, every manifold can be embedded in some $\R^n$ as a closed submanifold. Since closed submanifolds in $\R^n$ are complete w.r.t.\ the induced metric, we can define a complete Riemannian metric on $M$ by pulling back this metric.

We will denote the norm of a vector $v$ (w.r.t.\ $h$) by $|v|_h$. We also denote $|v|_g:=\sqrt{|g(v,v)|}$.
\end{remark}

\begin{remark}\rm
Since $M$ is globally hyperbolic, the proof of Theorem 3 in \cite{Bernard/Suhr} and the subsequent discussion, along with Corollary 1.8, show that there exists a smooth manifold $N$ and a diffeomorphism $M\cong \R\times N$, such that the projection $\tau:\R\times N\to \R, \ x\cong(t,z)\mapsto t,$ satisfies the following inequality:
\begin{align*}
d_x\tau(v)\geq \max\bigg\{2|v|_g,|v|_h\bigg\}
\end{align*}
for all causal vectors $v\in \C_x$. This function is called a \emph{splitting}.
\end{remark}

For further references on Lorentzian geometry, we refer the reader to \cite{ONeill}.

\section*{Acknowledgements}
I would like to thank my both advisors, Markus Kunze and Stefan Suhr, for all their support and guidance during my research. I truly appreciate the time they dedicated to our discussions and their willingness to help whenever I had questions.

\bibliography{Optimalpairs}
\end{document}